\documentclass[10pt]{amsart}
\usepackage{graphicx}
\usepackage[margin=1in]{geometry}
\usepackage{amsmath,amsthm,amsfonts,amssymb}
\usepackage{graphicx}
\usepackage{epsfig}
\usepackage{color}
\usepackage[spanish,english]{babel}
\usepackage[latin1]{inputenc}
\usepackage{amsmath}
\usepackage{amsfonts}
\usepackage{amssymb}
\usepackage{xcolor}
\newtheorem{thm}{Theorem}[section]
\newtheorem{cor}[thm]{Corollary}
\newtheorem{lemma}[thm]{Lemma}
\newtheorem{prop}[thm]{Proposition}
\theoremstyle{definition}
\newtheorem{defn}[thm]{Definition}
\theoremstyle{remark}
\newtheorem{remark}[thm]{Remark}
\numberwithin{equation}{section}

\usepackage{hyperref}
\begin{document}


\title[Existence and global second-order regularity for anisotropic parabolic equations]{Existence and global second-order regularity for anisotropic parabolic equations with variable growth}

\author[R. Arora]{Rakesh Arora}

\author[S. Shmarev]{Sergey Shmarev}
\address[R. Arora]{Department of Mathematical Sciences, Indian Institute of Technology (BHU), Varanasi 221005, India.}
\email{arora.npde@gmail.com}
\address[S. Shmarev]{Department of Mathematics, University of Oviedo, c/Federico Garc\'{i}a Lorca 18, 33007, Oviedo, Spain.}
\email{shmarev@uniovi.es}

\begin{abstract}
We consider the homogeneous Dirichlet problem for the anisotropic parabolic equation
\[
u_t-\sum_{i=1}^ND_{x_i}\left(|D_{x_i}u|^{p_i(x,t)-2}D_{x_i}u\right)=f(x,t)
\]
in the cylinder $\Omega\times (0,T)$, where $\Omega\subset \mathbb{R}^N$, $N\geq 2$, is a parallelepiped. The exponents of nonlinearity $p_i$ are given Lipschitz-continuous functions. It is shown that if $p_i(x,t)>\frac{2N}{N+2}$,
\[
\mu=\sup_{Q_T}\dfrac{\max_i p_i(x,t)}{\min_i p_i(x,t)}<1+\dfrac{1}{N}, \quad |D_{x_i}u_0|^{\max\{p_i(\cdot,0),2\}}\in L^1(\Omega),\quad f\in L^2(0,T;W^{1,2}_0(\Omega)),
\]
then the problem has a unique solution $u\in C([0,T];L^2(\Omega))$ with $|D_{x_i} u|^{p_i}\in L^{\infty}(0,T;L^1(\Omega))$, $u_t\in L^2(Q_T)$. Moreover,
\[
|D_{x_i}u|^{p_i+r}\in L^1(Q_T)\quad \text{with some $r=r(\mu,N)>0$},\qquad |D_{x_i}u|^{\frac{p_i-2}{2}}D_{x_i}u\in W^{1,2}(Q_T).
\]
The assertions remain true for a smooth domain $\Omega$ if $p_i=2$ on the lateral boundary of $Q_T$.

\medskip
	


\medskip
\end{abstract}

\thanks{The second author acknowledges the support of the Research Grant MCI-21-PID2020-116287GB-I00, Spain}

\today

\maketitle


\section{Introduction}
We consider the Dirichlet problem for the model anisotropic parabolic equation
\begin{equation}
\label{eq:main}
\begin{split}
& u_t-\sum_{j=1}^{N}D_j\left(|D_ju|^{p_j(z)-2}D_ju\right)=f(z)\quad \text{in $Q_T$},
\\
& \text{$u=0$ on $\partial\Omega\times (0,T)$},
\qquad
\text{$u(x,0)=u_0(x)$ in $\Omega$}.
\end{split}
\end{equation}
Throughout the text we denote by $z=(x,t)$ the points of the cylinder $Q_T=\Omega\times (0,T)$ with the base $\Omega$. The domain $\Omega\subset \mathbb{R}^N$, $N\geq 2$, is either a parallelepiped
\[
K_{\vec a}=\left\{x\in \mathbb{R}^N:\,x_i\in (-a_i,a_i),\,i=\overline{1,N}\right\}
\]
with the faces parallel to the coordinate planes and the edge lengths $2a_i$, or a domain with the smooth boundary $\partial\Omega\in C^k$, where $k\geq 2$ is a sufficiently large natural number. 
The assumptions on the exponents $p_i(z)$ differ according to the choice of the domain. The exponents of nonlinearity $p_i$ and the right-hand side $f$ are given functions whose properties will be specified later. The notation $D_j u$ is used for the partial derivative with respect to $x_j$, $D^2_{ij}u=D_{i}\left(D_{j}u\right)$, $i,j=\overline{1,N}$.

Equation \eqref{eq:main} with $p_i=p$ is sometimes called the equation of orthotropic diffusion \cite{Brasco-2021,Vazquez-2021-Bar-Arxiv}. It appears in the mathematical modeling of the diffusion processes where the diffusion rates are proportional to $|D_ju|^{p-2}$. In the present work, we are interested in the anisotropic case where the diffusion rates differ according to the directions $x_j$. The exponents $p_j>2$ correspond to the directions of slow diffusion, while $p_j<2$ means that the diffusion is fast. Since we allow $p_j$ to be functions of the variables $(x,t)$, it is possible that the character of diffusion in the $j$th direction changes from point to point.

At the points where $D_ju=0$ for some $j\in \{1,2,\ldots,N\}$, equation \eqref{eq:main} degenerates if $p_j>2$ or becomes singular if $p_j<2$. Despite the resemblance with the celebrated evolution $p$-Laplace equation

\begin{equation}
  \label{eq:p-Laplace}
u_t=\operatorname{div}\left(|\nabla u|^{p-2}\nabla u\right)=f,\qquad 1<p<\infty,
\end{equation}
which degenerates or becomes singular at the points where $|\nabla u|=0$, the properties of the solutions to equation \eqref{eq:main} are in striking contrast with the properties of the classical $p$-laplacian \eqref{eq:p-Laplace}. Unlike equation \eqref{eq:p-Laplace}, it may happen that the solutions of problem \eqref{eq:main} vanish in a finite time if the equation includes only one direction of fast diffusion with $p_i<2$. Conversely, the speed of propagation of disturbances may be finite or even zero in the direction of slow diffusion, see \cite{Vespri-J-Evol-Eq-2019,ant-shm-book-2015}.The difficulties brought in by the anisotropy and the nonhomogeneity of the diffusion operator are illustrated by the analysis of the self-similar solutions of Barenblatt type \cite{Vespri-2021-Bar,Vazquez-2021-Bar-Arxiv}. Unlike the isotropic case where the typical geometry is defined in terms of balls in $\mathbb{R}^N$, in the anisotropic case  it is defined by parallelepipeds with the edge lengths related to the exponents $p_i$.

In recent years, parabolic equations with anisotropic nonlinearity have been studied very actively. The theory of such equations, with constant or variable exponents $p_i$, includes theorems of existence and uniqueness of weak solutions,  properties of propagation of solutions in space and time, as well as certain regularity properties of solutions.
For the existence results for equation \eqref{eq:main}, as well as more general equations with the anisotropic principal part, nonlinear terms of lower order with variable nonlinearity, and under different regularity assumptions on the data, see, e.g., \cite{Ant-Shm-2009,Starovoitov-Tersenov-2010,ant-shm-book-2015,Qian_Yuan-2020,Bend-Kars-Saad-2013,Muk-2018,Boureanu-Velez-Santiago-2019,Mokhtari-2022,Chrif-2020, Chrif-2021,Ters-Ters-2019} and further references therein. The methods of proof in these works vary in dependence on the assumptions about the data and the exponents of nonlinearity.

The regularity of local solutions to equation \eqref{eq:main} has been studied by several authors. We refer, e.g., to \cite{Bahja-2021,Bahja-2019} for results on local continuity of solutions to equation \eqref{eq:main}. In \cite{Vespri-2019-Harnack,Vespri-2019-Holder,Vespri-2021-Harnack-Holder,Vespri-2021-Arxiv} the Harnack inequality and H\"older continuity of local solutions is established. It was recently proven in \cite{Brasco-2021} that the spatial gradient of the local solution of equation \eqref{eq:main} is bounded. The above results refer to the local solutions of equation \eqref{eq:main} and do not depend on the geometry of the boundary of the space domain $\Omega$. One of the key tools of the proofs are inequalities of the Caccippoli type, which prevents one from a straightforward extension of the regularity results to the whole of the cylinder $Q_T$.

In the present work, we are interested in the global regularity properties of solutions to problem \eqref{eq:main}. This issue has been recently studied  in several works. It is shown in \cite{Ters-Ters-2017} that problem \eqref{eq:main} with constant $p_i>1$ admits Lipschitz-continuous solutions. This is true if the domain $\Omega$ is either a parallelepiped, or is convex, $C^2$-smooth, and satisfies some geometric restrictions. The results of \cite{Ters-Ters-2017} are obtained under the following assumptions:
\[
f\equiv 0,\quad u_0\in C^2(\Omega),\quad \max_{\overline{\Omega}}\sum_{i=1}^N |D_i(|D_iu_0|^{p_i-2}D_iu_0)|<\infty.
\]
Moreover, it is shown that $D^2_{ij}u\in L^{2}(Q_T)$, provided that $p_i<2$. The key ingredient of the proof is the technique of ``doubling the space variables", which allows one to estimate the Lipschitz constant of the solution without differentiating the equation. Problem \eqref{eq:main} with the variable exponents $p_i \in C^{\alpha}[0,T]$, $\alpha\in (0,1)$, and the nonlinear source $f=f(x,t,u,\nabla u)$ is considered in \cite{Ters-Ters-2020}. Under the same assumptions on the geometry and smoothness of the space domain $\Omega$ as in \cite{Ters-Ters-2017}, the authors prove that if $f(x,t,u,\nabla u)$ satisfies certain growth conditions with respect to $u$ and  $\nabla u$, $u_0$ is Lipschitz-continuous, and for every $i=\overline{1,N}$ either $p_i(t)\geq 2$, or $p_i(t)\in (1,2]$ for all $t\in [0,T]$,  then the problem has a unique solution which is Lipschitz-continuous in the space variables. Moreover, if $p_i(t)\geq 2$ for all $i=\overline{1,N}$, then $u(x,\cdot,)\in C^{\frac{1}{2}}[0,T]$.

 Apart from these works, we are unaware of results on the global regularity for parabolic problems driven by the anisotropic operators, whilst the results for equations with the nonlinearity depending upon both space and time variables seem to be completely missing.  Addressing this issue, we establish global higher integrability and second-order regularity of solutions to the anisotropic parabolic equation \eqref{eq:main}. For problems of this type, an extension of previous contributions is particularly delicate due to varying anisotropy in both space and time variables.
 Let us describe the results of the present work. First we consider problem \eqref{eq:main} in a rectangular domain $\Omega$ with the data satisfying the conditions

\[
f\in L^2(0,T;W^{1,2}_0(\Omega)),\qquad \sum_{i=1}^N\int_\Omega |D_iu_0|^{\max\{2,p_i(x,0)\}}\,dx<\infty.
\]
The exponents $p_i$ are Lipschitz continuous functions and satisfy conditions \eqref{eq:p-i-1}, \eqref{eq:p-i-2}. The latter condition means that at every point $z\in Q_T$ the maximum and minimum of $p_i(z)$ are sufficiently close. The following is a brief account of the results in the case when $\Omega$ is a parallelepiped:

\begin{itemize}
\item problem \eqref{eq:main} has a unique solution $u\in C([0,T];L^2(\Omega))$ such that

\[
u_t\in L^2(Q_T),\qquad u\in L^{\infty}(0,T;W^{1,2}_0(\Omega)),\qquad D_iu \in L^{\infty}(0,T; L^{p(z)}(\Omega)),\quad i=\overline{1,N};
\]

\item the solution possesses the property of global higher integrability of the partial derivatives:

    \[
    \int_{Q_T}|D_i u|^{p_i(z)+r}\,dz<\infty,\quad r\in (0,r^\ast),\quad i=\overline{1,N},
    \]
    with some $r^\ast$ depending on the space dimension $N$ and the number $\sup_{Q_T}\dfrac{\max_ip_i(z)}{\min_i p_i(z)}$;

    \item there exist the second-order derivatives

    \[
    |D_ju|^{\frac{p_j(z)-2}{2}}D_ju\in W^{1,2}(Q_T),\qquad j=\overline{1,N};
    \]

    \item the same existence and regularity results hold for the regularized anisotropic equation

    \begin{equation}
    \label{eq:reg-intro}
    u_t - \sum_{j=1}^{N}D_j\left((\epsilon^2+|D_ju|^2)^{\frac{p_j(z)-2}{2}}D_j u\right)=f,\qquad \epsilon\in (0,1);
    \end{equation}

    \item it is not required that any of the exponents $p_i(z)$ belongs to the range corresponding to fast or slow diffusion; each of $p_i(z)$ may vary within the interval $\left(\frac{2N}{N+2},p_h^\ast(z)\right)$, where the critical exponent $p_h^\ast(z)$ is defined in \eqref{eq:crit-h}.
\end{itemize}

The results are extended to the solutions of the problem posed in a smooth domain.
However, in this case we additionally assume that $p_i(z)=2$ on $\partial\Omega\times [0,T]$.

\section{Assumptions and main results}
To formulate the results we have to introduce the function spaces the solution of problem \eqref{eq:main} belongs to. For a given vector $\vec p(z)=(p_1(z),\ldots,p_N(z))$ with measurable and bounded components defined on $Q_T$, $p_i(z)>1$ in $Q_T$, we define the functions

\begin{equation}
\label{eq:harm-mean}
p_h(z)=\dfrac{N}{\sum_{i=1}^N\frac{1}{p_i(z)}},\quad \text{the harmonic mean of $p_1(z),\ldots,p_N(z)$},
\end{equation}
and

\begin{equation}
\label{eq:crit-h}
p_h^{\ast}(z)=\begin{cases}\dfrac{Np_h(z)}{N-p_h(z)} & \text{if $N>p_h(z)$},
\\
\text{any number from $(1,\infty)$} & \text{if $N\leq p_h(z)$}.
\end{cases}
\end{equation}
Assume that

\begin{equation}
\label{eq:p-i-1}
p_i:Q_T\mapsto (1,\infty),\qquad \text{$\displaystyle \frac{2N}{N+2}<p_i(z)<p_h^\ast(z)$ in $\overline{Q}_T$, \quad $i=\overline{1,N}$}.
\end{equation}

Given a measurable function $q:\Omega \mapsto (1,\infty)$, let $L^{q(\cdot)}(\Omega)$ be the linear space

\begin{equation}
\label{eq:Lebesgue}
L^{q(\cdot)}(\Omega)=\left\{\text{$u$ is measurable on $\Omega$}:\,\rho_{q(\cdot)}(u)\equiv \int_{\Omega}|u|^{q(x)}\,dx<\infty \right\}.
\end{equation}
The space $L^{q(\cdot)}(\Omega)$ equipped with the norm

\[
\|u\|_{q(\cdot),\Omega}=\inf\left\{\lambda>0: \rho_{q(\cdot)}\left(\frac{u}{\lambda}\right)\leq 1\right\}
\]
is a Banach space. For a vector $\vec q(x)=(q_1(x),\ldots,q_N(x))$ with the components satisfying conditions \eqref{eq:p-i-1} in $\Omega$, we define the variable anisotropic Sobolev space

\begin{equation}
\label{eq:main-space-1}
\begin{split}
W^{1,\vec q(\cdot)}_0(\Omega) & =\left\{\text{the closure of $C_0^\infty(\Omega)$ w.r.t. the norm}
\quad \|u\|_{W^{1,\vec q(\cdot)}_0(\Omega)}=\sum_{i=1}^N\|D_iu\|_{p_i(\cdot),\Omega}\right\}.
\end{split}
\end{equation}
The equivalent definitions of these spaces and the main properties of their elements are discussed in Section \ref{sec:spaces}. To deal with the functions defined on the cylinder $Q_T$, we introduce the spaces of functions depending on $x$ and $t$. For a vector $\vec p:\,\Omega\times (0,T)=Q_T\mapsto \mathbb{R}^N$ with the components satisfying conditions \eqref{eq:p-i-1}

\begin{equation}
\label{eq:space-parab}
\begin{split}
& \mathbb{V}_t(\Omega)=W^{1,\vec p(\cdot,t)}_0(\Omega),\quad \text{for a.e. in $t\in (0,T)$},
\\
& \mathbb{W}(Q_T)=\left\{u: (0,T)\mapsto \mathbb{V}_t(\Omega):\,u\in L^2(Q_T),\,|D_i u|^{p_i(z)}\in L^1(Q_T),\ i=\overline{1,N}\right\}.
\end{split}
\end{equation}
The space $\mathbb{W}(Q_T)$ is equipped with the norm
\[
\|u\|_{\mathbb{W}}=\|u\|_{2,Q_T}+\sum_{i=1}^N\|D_iu\|_{p_i(\cdot),Q_T}.
\]
We will also need the following functions:

\begin{equation}
\label{eq:vee}
p^\vee(z)=\max\{p_1(z),\ldots,p_N(z)\},\qquad p^\wedge(z)=\min\{p_1(z),\ldots,p_N(z)\}.
\end{equation}

\begin{defn}
\label{def:main}
A function $u$ is called weak solution of problem \eqref{eq:main} if

\begin{itemize}
\item[(i)] $u\in C^0([0,T];L^2(\Omega))\cap \mathbb{W}(Q_T)$, $u_t\in L^{2}(Q_T)$,

\item[(ii)] for every $\phi\in \mathbb{W}(Q_T)$

\begin{equation}
\label{eq:def-main}
\int_{Q_T}\left(u_t\phi+\sum_{i=1}^{N}|D_iu|^{p_i(z)-2}D_iu D_i\phi\right)\,dz=\int_{Q_T}f\phi\,dz,
\end{equation}

\item[(iii)] for every $\phi\in L^2(\Omega)$ $(u(x,t)-u_0(x),\phi)_{2,\Omega}\to 0$ as $t\to 0^+$.
\end{itemize}
\end{defn}

The main results of the work are given in the following assertions.

\begin{thm}
\label{th:existence}
Let $\Omega=K_{\vec a}$ be a parallelepiped. Assume that the vector $\vec p(z)$ satisfies conditions \eqref{eq:p-i-1} and

\begin{equation}
\label{eq:p-i-2}
p_i \in C^{0,1}(\overline{Q}_T),
\qquad
\mu=\sup_{Q_T}\frac{p^\vee(z)}{p^\wedge(z)}<1+\frac{1}{N}.
\end{equation}
Then for every $u_0\in W^{1,2}_0(\Omega)\cap W^{1,\vec p(\cdot,0)}_0(\Omega)$ and $f\in L^2(0,T; W^{1,2}_0(\Omega))$ problem \eqref{eq:main} has a unique weak solution $u \in \mathbb{W}(Q_T)$. Moreover,

\[
u\in L^{\infty}(0,T;W^{1,2}_0(\Omega)\cap W^{1,\vec p(\cdot)}_0(\Omega)),\qquad u_t\in L^2(Q_T)
\]
with

\begin{equation}
\label{eq:main-est}
\begin{split}
\displaystyle \|u_t\|_{2,Q_T} & +\operatorname{ess}\sup\limits_{(0,T)}\|u\|_{W^{1,2}(\Omega)}+ \operatorname{ess}\sup\limits_{(0,T)}\|u\|_{W_0^{1,\vec p(\cdot, t)}(\Omega)}
\\
& \leq C\left(1+\|f\|_{L^2(0,T;W_0^{1,2}(\Omega))}+ \|u_0\|_{W_0^{1,2}(\Omega)}+ \|u_0\|_{W_0^{1,\vec p(x,0)}(\Omega)}\right).
\end{split}
\end{equation}
\end{thm}

\begin{remark}
The conditions of Theorem \ref{th:existence} allow the exponents $p_j(z)$ to vary within the interval $\left(\frac{2N}{N+2},p_h^\ast\right)$. This means that the diffusion rate in the $j$th direction depends on the point $z\in Q_T$ and may be slow on a part of the domain, {\it i.e.} $p_j(z)\geq 2$, and fast on its complement where $p_j(z)<2$. The second condition in \eqref{eq:p-i-2} can be relaxed if the diffusion type in each direction does not change on the whole of the domain. Relabelling the directions, we may assume $p_i(z) \geq 2$ in $\overline{Q}_T$ for $i= \overline{1, K}$ (slow diffusion) and $p_i(z)\leq 2-\sigma$ for $i=\overline{K+1,N}$ with some $\sigma>0$ (fast diffusion). In this case, Theorem \ref{th:existence} can be proven under the weaker gap condition in \eqref{eq:p-i-2} (see Remarks \ref{rem:integrability-1}, \ref{rem:improved}, \ref{rem:improved-cylinder}):
\begin{equation}
\notag
\label{eq:gap-improved}
\mu=\sup_{Q_T}\frac{p^\vee(z)}{p^\wedge(z)}<1+\frac{2}{N}.
\end{equation}
\end{remark}

\begin{remark}\label{eq:second-der:fast}
In case of fast diffusion in $j$th direction, {\it i.e.} $p_j(z) < 2$ in $Q_T$, we have the following inclusion:
\[
D^2_{ij} u \in L^{p_j(\cdot)}(Q_T) \quad \text{for all} \quad i=\overline{1, N}
\]
- see Remarks \ref{eq:second-der:fast-0} and \ref{eq:second-der:fast-1}
\end{remark}

\begin{thm}
\label{th:high-order-reg}
Let the conditions of Theorem \ref{th:existence} be fulfilled and $u$ be the weak solution of problem \eqref{eq:main}.

\begin{itemize}
\item[(i)] The solution has the property of global higher integrability: for every $i=\overline{1,N}$

\begin{equation}
\label{eq:high-int}
\int_{Q_T}|D_iu|^{p_i(z)+r}\,dz<\infty\quad \text{for every $r\in (0,r^\ast)$},\quad r^\ast= \dfrac{4-2N(\mu-1)}{N+2}.
\end{equation}

\item[(ii)] The solution has the second-order derivatives in the following sense:
\begin{equation}
\label{eq:second-der}
|D_iu|^{\frac{p_i(z)-2}{2}}D_iu\in W^{1,2}(Q_T),\qquad i=\overline{1,N}.
\end{equation}
\end{itemize}
\end{thm}

\begin{remark}
The assertions of Theorem \ref{th:high-order-reg} remain true if $p_i(z)\geq 2$ for all $i=\overline{1,N}$ in $\overline{Q}_T$ and $\mu<1+\dfrac{2}{N}$ - see Remarks \ref{rem:improved} and \ref{rem:improved-cylinder}.
\end{remark}

\begin{remark}
The property of higher integrability of the gradient is well-known for the solutions of the isotropic parabolic equations with $(p,q)$-growth, see, e.g., \cite{Bog-Duz-Marc-2013,A-S-2020}. Let $p_i(z)=p(z)$ for all $i=\overline{1,N}$. In this special case  $p^\vee(z)=p^\wedge(z)=p(z)$, $\mu=1$, and Theorem \ref{th:high-diff-reg} (i) recovers the maximal possible gap between the exponents $p$ and $q$ found in \cite{Bog-Duz-Marc-2013} for the case of isotropic slow diffusion: $2\leq p\leq q<p+\frac{4}{N+2}$.
\end{remark}

\begin{remark}
The assertions of Theorems \ref{th:existence} and \ref{th:high-order-reg} hold true for the solutions of the regularized equation \eqref{eq:reg-intro}. Moreover, the study of the regularized equation constitutes the bulk of the rest of the work. The conclusions of Theorems \ref{th:existence},  \ref{th:high-order-reg} follow by passing to the limit as $\epsilon\to 0$ in the corresponding results for the solutions of the regularized problem given in Theorems \ref{th:existence-reg}, \ref{th:high-diff-reg}.
\end{remark}

The next result addresses the situation where the domain $\Omega$ is smooth. Following the same scheme of arguments, we show that the assertions of Theorems \ref{th:existence} and \ref{th:high-order-reg} remain true but under additional restrictions on the anisotropy of the diffusion operator.

\begin{thm}
\label{th:smooth-domain}
Let $\Omega$ be a bounded domain with $\partial\Omega\in C^k$ with
\[
k\geq 1+N\left(\frac{1}{2}-\frac{1}{p^+}\right),\qquad p^+=\sup_{Q_T}p^\vee(z).
\]
Assume that $p_i(z)=2$ on $\partial\Omega\times [0,T]$. If   $\vec p(z)$ satisfies conditions \eqref{eq:p-i-1}, $p_i \in C^{0,1}(\overline{Q}_T)$, and
\[
\mu=\sup_{Q_T}\frac{p^\vee(z)}{p^\wedge(z)}<1+\frac{1}{N},\quad \text{or}\quad  \text{$p_i(z)\geq 2$ in $\overline{Q}_T$ and $\mu<1+\dfrac{2}{N}$},
\]
then for every $u_0\in W^{1,2}_0(\Omega)\cap W^{1,\vec p(\cdot,0)}_0(\Omega)$ and $f\in L^2(0,T; W^{1,2}_0(\Omega))$ problem \eqref{eq:main} has a unique weak solution $u \in \mathbb{W}(Q_T)$ such that

\[
u\in L^{\infty}(0,T;W^{1,2}_0(\Omega)\cap W^{1,\vec p(\cdot)}_0(\Omega)),\qquad u_t\in L^2(Q_T).
\]
The solution $u$ satisfies estimate \eqref{eq:main-est}. Moreover, the solution possesses the property of higher integrability of the gradient \eqref{eq:high-int}, and inclusions \eqref{eq:second-der} hold.
\end{thm}%

Let us outline the contents of the work. In Section \ref{sec:spaces} we introduce the variable Lebesgue and Sobolev spaces and collect the known results on the anisotropic Sobolev spaces used in the rest of the work. The rectangular domains are natural for the anisotropic spaces because they allow one to extend a given function to a broader domain or the whole space by a function from the same anisotropic space. This is not always possible in a smooth domain because the standard procedure based on rectifying the boundary portion mixes the partial derivatives, which have different orders of integrability. This difficulty may be overcome by considering the class of domains with ``$p(\cdot)$-extension property", i.e., the domains for which such an extension is possible without altering the anisotropic space, see \cite{Boureanu-Velez-Santiago-2019}. A parallelepiped is one of the known examples of anisotropic $\vec p(\cdot)$-extension domains, although the complete characterization of this class is not available thus far.

In Section \ref{sec:density} we construct the basis for the anisotropic variable Sobolev space in the rectangular and smooth domains. In both cases we take for the basis the set of eigenfunctions of the Dirichlet problem for the Laplace operator and show that it is dense in the anisotropic variable Sobolev space.

In Section \ref{sec:regularization} the regularized nondegenerate problems for equation \eqref{eq:reg-intro} in a rectangular domain are formulated. The regularized problems are solved with the method of Galerkin in the basis constructed in Section \ref{sec:density}. Section \ref{sec:estimates} is entirely devoted to deriving a priori estimates for the solutions of the finite-dimensional approximate problems. The global regularity of the basis functions allows one to obtain global uniform estimates on the higher-order derivatives of the approximate solutions. These estimates entail the regularity of the sought solution, which is obtained later as the limit of the sequence of approximations.

It is proven in Section \ref{sec:higher-integr} that the partial derivatives of the finite-dimensional approximations are integrable in $Q_T$ with the orders $p_i(z)+\delta$ with some $\delta>0$, instead of the orders $p_i(z)$ prompted by the equation. The integrals are bounded by a constant that does not depend of $\epsilon$ and the number of the approximation. To derive these estimates we prove a special anisotropic interpolation inequality and combine it with the uniform a priori estimates on the second-order derivatives obtained in Section \ref{sec:estimates}.

The main results are proven in Section \ref{sec:proofs}. We prove first the existence and regularity results for the solution of the regularized problem \eqref{eq:reg-intro}. The proof of the existence theorem relies on the weak and strong compactness of the sequence of the approximate solutions, which follow from the uniform a priori estimates. The limits of the nonlinear terms are identified by monotonicity. Moreover, the property of higher integrability of the spatial derivatives allows one to prove the strong and pointwise convergence of the gradients, which is used then in the proof of the second-order regularity of the obtained solution. The proofs of Theorems \ref{th:existence}, \ref{th:high-order-reg} follows the same scheme with the difference that now we have to pass to the limit as $\epsilon\to 0$ in the family of solutions of equation \eqref{eq:reg-intro}, which requires an additional step in the arguments.

In Section \ref{sec:smooth-domain} the results are extended to the case of a smooth domain with the boundary $\partial\Omega\in C^k$. It turns out that such an extension is possible if at every point of the lateral boundary of the cylinder $Q_T$ the flux vector with the components $|D_iu|^{p_i(z)-2}D_iu$ either equals zero, or points in the direction of the normal to the boundary. The latter is true either if all $p_i(z)=2$ on the boundary, or if the normal vector has only one nonzero component, which means that $\Omega$ is a rectangular domain. In the isotropic equation \eqref{eq:p-Laplace} this restriction does not appear because the flux $|\nabla u|^{p(z)-2}\nabla u$ is always proportional to $\nabla u$.

\section{The function spaces}
\label{sec:spaces}
\subsection{Variable Lebesgue spaces}
A thorough insight into the theory of variable Lebegue and Sobolev space can be found in the monograph \cite{DHHR-2011}. Here we confine ourselves to presenting only the properties of the Lebesgue spaces \eqref{eq:Lebesgue} needed in this work. Let $\Omega\subset \mathbb{R}^N$ be a bounded Lipschitz domain, and $p:\Omega\mapsto \mathbb{R}$ be a measurable function with values in an interval $[p^-,p^+]\subset (1,\infty)$, $p^\pm=const$. The space $L^{p(\cdot)}(\Omega)$ is defined by the modular

\[
\rho_{p(\cdot)}(u)=\int_{\Omega}|u|^{p(x)}\,dx
\]
The dual space of $L^{p(\cdot)}(\Omega)$ is the space $L^{p'(\cdot)}(\Omega)$ with the conjugate exponent $p'(x)=\frac{p(x)}{p(x)-1}$.

The generalized H\"older inequality holds: for every $f\in L^{p(\cdot)}(\Omega)$ and $g\in L^{p'(\cdot)}(\Omega)$

\begin{equation}
\label{eq:Holder}
\int_\Omega| f g|\,dx\leq 2\|f\|_{p(\cdot),\Omega}\|g\|_{p'(\cdot),\Omega}.
\end{equation}

If $p,q$ are measurable functions in $\Omega$ and $1<p(x)\leq q(x)<\infty$ a.e. in $\Omega$, then the embedding $L^{q(\cdot)}(\Omega)\subset L^{p(\cdot)}(\Omega)$ is continuous and

\[
\|u\|_{p(\cdot),\Omega}\leq C\|u\|_{q(\cdot),\Omega}.
\]
The relation between the modular and the norm of $L^{p(\cdot)}(\Omega)$ is given by the following inequalities:

\begin{equation}
\label{eq:norm-mod}
\min\left\{\|u\|^{p^-}_{p(\cdot),\Omega}, \|u\|^{p^+}_{p(\cdot),\Omega}\right\} \leq \rho_{p(\cdot)}(u) \leq \max\left\{\|u\|^{p^-}_{p(\cdot),\Omega}, \|u\|^{p^+}_{p(\cdot),\Omega}\right\}.
\end{equation}

The set $C_0^\infty(\Omega)$ is dense in $L^{p(\cdot)}(\Omega)$.

\subsection{Anisotropic variable Sobolev spaces}
Let $p_0, p_1, \ldots, p_N$ be measurable functions defined on $\Omega$ and $\vec{p}(x)=(p_1(x),\ldots,p_N(x))$ be a vector. By $C_{\rm log}(\overline{\Omega})$ we denote the set of functions continuous in $\overline{\Omega}$ with a logarithmic modulus of continuity:
\begin{equation}
\label{eq:log-cont}
|q(x)-q(y)|\leq \omega(|x-y|), \quad \forall x,y\in \overline{\Omega}, \;\;\vert x-y\vert <\frac{1}{2}.
\end{equation}
 where {$\omega$} is a nonnegative function such that
\[
\limsup_{s\to 0^+}\omega(s)\ln \frac{1}{s}=C,\quad C=const.
\]
Let us assume that
\[
p_i \in C_{{\rm log}}(\overline{\Omega}),\quad p_i(x)\in \left[p^-,p^+\right]\subset (1,\infty),\quad  i=\overline{1,N}.
\]
Let $p_h$, $p_h^\ast$, $p^\vee$, $p^\wedge$ be the functions defined in \eqref{eq:harm-mean}, \eqref{eq:crit-h}, \eqref{eq:vee}.
Apart from the space $W^{1,\vec p(\cdot)}_0(\Omega)$ introduced in \eqref{eq:main-space-1}, we consider the following spaces:

\begin{itemize}

\item[(i)] $\qquad \displaystyle \begin{cases}
& W^{1,(p_0(\cdot),\vec{p}(\cdot))}(\Omega)=\left\{u\in L^{p_0(\cdot)}(\Omega):\;D_iu\in L^{p_i(\cdot)}(\Omega), \  i=\overline{1,N} \right\},
\\
& \|u\|_{W^{1,(p_0,\vec p)}(\Omega)}=\|u\|_{p_0(\cdot),\Omega}+\sum_{i=1}^{N}\|D_iu\|_{p_i(\cdot),\Omega},
\end{cases}$

\bigskip

\item[(ii)]
$ \qquad \displaystyle
\begin{cases}
& W^{1,(p^{\vee},\vec{p}(\cdot))}(\Omega)=\{u\in L^{p^\vee(\cdot)}(\Omega):\,D_iu\in L^{p_i(\cdot)}(\Omega),  i=\overline{1,N}\},
\\
&
\|u\|_{W^{1,(p^\vee(\cdot),\vec p(\cdot))}(\Omega)}=\|u\|_{p^\vee(\cdot),\Omega}+\sum_{i=1}^{N}\|D_iu\|_{p_i(\cdot),\Omega},
\end{cases}
$

\bigskip

\item[(iii)]$\stackrel{\circ}{W}^{1,(p_0(\cdot),\vec{p}(\cdot))}(\Omega) =W^{1,(p_0(\cdot),\vec{p}(\cdot))}(\Omega)\cap W^{1,1}_0(\Omega)$,

    \noindent $\stackrel{\circ}{W}^{1,(p^\vee(\cdot),\vec{p}(\cdot))}(\Omega)= W^{1,(p^\vee(\cdot),\vec{p}(\cdot))}(\Omega)\cap W^{1,1}_0(\Omega)$,

  \bigskip

\item[(iv)] $W_0^{1,(p^\vee(\cdot),\vec p(\cdot))}(\Omega)$ = $\left\{\text{the closure of $C_0^\infty(\Omega)$ w.r.t. the norm of $W^{1,(p^\vee(\cdot),\vec p(\cdot))}(\Omega)$}\right\}$.
\end{itemize}

\subsection{Preliminaries}
\label{subsec:domain}
By a rectangular domain $\Omega$ we mean a parallelepiped $K_{\vec a}$. The boundary 
of a rectangular domain $\Omega$ is represented in the form $\partial\Omega=\Gamma_0\cup \Gamma$, where

\begin{enumerate}
\item $\Gamma$ is composed of $(N-1)$-dimensional open sets $\Gamma_i$, which are the faces of $\Gamma$ and lay in the coordinate planes $x_i=\pm a_i$,

    \item $\Gamma_0$ contains the edges and vertices and has the surface measure zero.
\end{enumerate}

\begin{prop}[\cite{Fan-2011}, Th.2.4 and \cite{Zhikov-2011-density}, Sec.13]
\label{pro:Fan-1}
Let $\Omega\subset \mathbb{R}^N$ be a bounded domain with Lipschitz boundary. If $p_i \in C_{\rm log}(\overline{\Omega})$, then

\begin{enumerate}
\item $C_0^\infty(\Omega)$ is dense in $\stackrel{\circ}{W}^{1,(p^\vee(\cdot),\vec{p}(\cdot))}(\Omega)$ and, thus,  $\stackrel{\circ}{W}^{1,(p^\vee(\cdot),\vec{p}(\cdot))}(\Omega)=W_0^{1,(p^\vee(\cdot),\vec p(\cdot))}(\Omega)$;

    \item $C^\infty(\Omega)$ is dense in  $W^{1,(p^\vee(\cdot),\vec p(\cdot))}(\Omega)$ if $\Omega$ is a rectangular domain.
\end{enumerate}
\end{prop}

\begin{prop}[\cite{Fan-2011}, Th.2.5]
  \label{pro:compact-embedding}
  Let $\Omega$ be a rectangular
  domain and $\vec p \in C^{0}(\overline \Omega)^N$.
  If $q\in C^0(\overline{\Omega})$ and

  \[
  q(x)<\max\left\{p^\vee(x),p_h^\ast(x)\right\}\quad \text{for all $x\in \overline{\Omega}$},
  \]
  then

  \[
  W^{1,(p^\vee(\cdot),\vec p(\cdot))}(\Omega)\hookrightarrow L^{q(\cdot)}(\Omega)\quad \text{(compact embedding)}.
  \]
\end{prop}

\begin{prop}[\cite{Fan-2011}, Th.2.6]
\label{pro:Poincare-like}
Let $\Omega$ be a bounded domain and $p_i \in C^0(\overline{\Omega})$, $i=\overline{1, N}$. If

\begin{equation}
\label{eq:exp}
p^\vee(x)<p_h^\ast(x)\quad \text{for all $x\in \overline{\Omega}$,}
\end{equation}
then

\begin{equation}
\label{eq:almost-Poincare}
\|u\|_{p^\vee(\cdot),\Omega}\leq C\sum_{i=1}^N\|D_iu\|_{p_i(\cdot),\Omega}\quad \text{for all $u\in W_0^{1,(p^\vee(\cdot),\vec p(\cdot))}(\Omega)$}
\end{equation}
with a constant $C$ independent of $u$. Hence, under condition \eqref{eq:exp} the functional $\sum_{i=1}^N\|D_iu\|_{p_i(\cdot),\Omega}$ defines an equivalent norm of $W_0^{1,(p^\vee(\cdot),\vec p(\cdot))}(\Omega)$.
\end{prop}

\begin{prop}[\cite{Fan-2011}, Th.2.7]
\label{pro:equiv-spaces}
\begin{enumerate}
\item Let $\Omega$ be a rectangular 
domain and $p_i \in C^0(\overline{\Omega})$, $i=\overline{1,N}$. If \eqref{eq:exp} holds, then
    \[
    W^{1,(p^\vee(\cdot),\vec p(\cdot))}(\Omega)=W^{1,(p_0(\cdot),\vec p(\cdot))}(\Omega)
    \]
    for any $p_0 \in C^0(\overline{\Omega})$ satisfying $p_0(x)<p^\ast_h(x)$ everywhere in $\overline{\Omega}$.

    \item Let $\Omega$ be a bounded domain, $p_i \in C^0(\overline{\Omega})$, $i=\overline{1,N}$. If \eqref{eq:exp} holds, then
        \[
        \text{$W_0^{1,(p^\vee(\cdot),\vec p(\cdot))}(\Omega)=W_0^{1,(1,\vec p(\cdot))}(\Omega)$, and $\stackrel{\circ}{W}^{1,(p^\vee(\cdot),\vec p(\cdot))}(\Omega) = \stackrel{\circ}{W}^{1,(1,\vec p(\cdot))}(\Omega)$.}
        \]
\end{enumerate}
\end{prop}

\begin{prop}[\cite{Fan-2011}, Th.2.8]
\label{pro:main-space}
Let $\Omega$ be a bounded domain, $p_i \in C^0(\overline{\Omega})$, $i=\overline{1, N }$. If \eqref{eq:exp} is fulfilled, then

\[
W^{1,\vec p(\cdot)}_0(\Omega)=W_0^{1,(p^\vee(\cdot),\vec p(\cdot))}(\Omega).
\]
\end{prop}

\section{Dense sets in anisotropic Sobolev spaces}
\label{sec:density}

We distinguish between the cases where $\Omega$ is a rectangular domain, or has the smooth boundary.

\subsection{The rectangular domain $K_{\vec a}$}
 Let us consider first the case $\Omega=K_{\vec a}$. The eigenfunctions of the Dirichlet problem for the Laplace operator

\begin{equation}
\label{eq:eigen-Dir}
\text{$\Delta \psi_{\mathbf k}+\lambda_{\mathbf k}\psi_{\mathbf k}=0$ in $K_{\vec a}$},\qquad \text{$\psi_{\mathbf k}=0$ on $\Gamma$},
\end{equation}
form an orthogonal basis of $L^2(K_{\vec a})$. The solutions of problem \eqref{eq:eigen-Dir} have the form

\begin{equation}
\label{eq:phi-explicit}
\psi_{\mathbf{k}}(x)=C\prod_{i=1}^N\sin\left(\frac{\pi k_i x_i}{a_i}\right),\quad \lambda_{\mathbf k}= \pi^2\sum_{i=1}^N\frac{k_i^2}{a_i^2}
\end{equation}
where $\mathbf{k}=(k_1,\ldots,k_N)$, $k_i\in \mathbb{N}$, $C=const>0$ is the normalizing constant.
It follows that $\Delta\psi_{\mathbf k}=0$ on $\Gamma=\partial\Omega\setminus \Gamma_0$. For all ${\mathbf k}\in \mathbb{N}^N$ and $s\in \mathbb{N}$

\[
\text{$\Delta^s\psi_{\mathbf k}+(-\lambda_{\mathbf k})^s\psi_{\mathbf k}=0$ in $\Omega$},\qquad \text{$\Delta^{s-1}\psi_{\mathbf k}=0$ on $\Gamma=\partial\Omega\setminus \Gamma_0$}.
\]
The set of eigenpairs $(\psi_{\mathbf k},\lambda_{\mathbf k})$ can be reordered and presented in the form

\[
\{\psi_i\}_{i=1}^{\infty}, \quad 0 < \lambda_1\leq \lambda_2\leq \ldots\leq \lambda_i \to \infty \;\;\text{as $i\to \infty$},
\]
where $\psi_i$ corresponding to different $\lambda_i$ are orthogonal in $L^2(\Omega)$, but the same eigenvalue $\lambda_i$ may correspond to various $\psi_i$. The eigenfunctions are normalized by the condition $\|\psi_i\|_{2,\Omega}=1$.

Fix $m=1,2,\ldots$. For every multi-index $\alpha\in \mathbb{R}^N$ and any vector $\mathbf{k}$, $k_i=1,2,\ldots$,
\[
\left|D^\alpha\psi_{\mathbf k}\right|= C\pi^{2m}\prod_{i=1}^N \left(\frac{k_i}{a_i}\right)^{\alpha_i}\Phi_i(x_i),
\]
where
\[
\Phi_i(s)=\begin{cases}
\left|\sin\left(\frac{\pi k_i s}{a_i}\right)\right| & \text{if $\alpha_i=0$ or $\alpha_i$ is an even number},
\\
\left|\cos\left(\frac{\pi k_i s}{a_i}\right)\right| & \text{if $\alpha_i$ is an odd number}.
\end{cases}
\]

It is straightforward to check that for $|\alpha|=2m$

\[
\|D^\alpha \psi_{\mathbf k}\|_{2,\Omega}^2= C^2\pi^{4m}\prod_{i=1}^N \left(\frac{k_i}{a_i}\right)^{2\alpha_i}\|\Phi_i\|^2_{2,(-a_i,a_i)}=\pi^{4m} \prod_{i=1}^N \left(\frac{k_i}{a_i}\right)^{2\alpha_i}.
\]
Let us denote $k_+=\max_i k_i$, $a_+=\max_i a_i$, $a_-=\min_i a_i$. Then

\[
\prod_{i=1}^N \left(\frac{k_i}{a_i}\right)^{2\alpha_i}\leq
\leq \frac{a_+^{4m}}{a_-^{4m}}\left(\sum_{i=1}^N \frac{k_i^2}{a_i^2}\right)^{2m}=\frac{1}{\pi^{4m}}\frac{a_+^{4m}}{a_-^{4m}} \left(\pi^2\sum_{i=1}^{N}\frac{k_i^2}{a_i^2}\right)^{2m}= \left(\frac{a_+}{\pi a_-}\right)^{4m} \|\Delta^m \psi_{\mathbf{k}}\|^2_{2,\Omega}.
\]
Thus, for every $m=0,1,2,\ldots$ there exists a constant $C=C(a_\pm,N,M,m)$ such that
\begin{equation}
\label{eq:mixed-eigenfunctions-even}
\|D^\alpha \psi_{\mathbf k}\|_{2,\Omega}^2\leq C \|\Delta^m \psi_{\mathbf{k}}\|^2_{2,\Omega},\quad |\alpha|=2m.
\end{equation}
Let $|\alpha|=2m-1$, $m=1,2,\ldots$, be an odd number. Since $k_+\geq 1$, in this case we have
\[
\prod_{i=1}^N \left(\frac{k^2_i}{a_i^2}\right)^{\alpha_i}
\leq
\frac{k_+^{2(2m-1)}}{a_-^{2(2m-1)}}
\leq a_-^{2}\left(\frac{a_+}{a_-}\right)^{4m} \frac{k_+^{4m}}{a_+^{4m}}\leq a_-^2\left(\frac{a_+}{\pi a_-}\right)^{4m} \left(\pi^2\sum_{i=1}^N \frac{k_i^2}{a_i^{2}}\right)^{2m} = C(a_\pm,m) \|\Delta^m \psi_{\mathbf{k}}\|^2_{2,\Omega},
\]

whence
\begin{equation}
\label{eq:mixed-eigenfunctions-odd}
\|D^\alpha \psi_{\mathbf k}\|_{2,\Omega}^2\leq C' \|\Delta^m \psi_{\mathbf{k}}\|^2_{2,\Omega},\quad |\alpha|=2m-1.
\end{equation}

\subsection{Density of $\{\psi_i\}_{i=1}^{\infty}$ in $W_0^{1,\vec{p}(\cdot)}(\Omega)$}

Consider the Hilbert space

\[
H^m(\Omega)=\{u: D^\alpha u\in L^2(\Omega), \,\alpha=(\alpha_1,\ldots,\alpha_N),\,|\alpha|=\sum\alpha_i\leq m\}
\]
equipped with the usual norm

\[
\|u\|_{H^m(\Omega)}=\sum_{0\leq |\alpha|\leq m}\|D^\alpha u\|_{2,\Omega}.
\]

\begin{prop}
\label{pro:by-parts}
Let $\Omega$ be a rectangular domain. For every $v\in C_0^\infty(\Omega)$

\[
\int_{\Omega}|\Delta v|^2\,dx=\sum_{i,j=1}^N\int_{\Omega}|D^2_{ij}v|^2\,dx.
\]
\end{prop}

\begin{proof} Integrating by parts two times we obtain

\[
\begin{split}
\int_{\Omega}|\Delta v|^2\,dx & = \int_{\Omega}\operatorname{div}(\nabla v)\Delta v\,dx
=\sum_{i=1}^N\int_{\Gamma_i}(\vec{\nu}\cdot \nabla v) \Delta v\,dS-\sum_{i=1}^N\int_\Omega D_iv\operatorname{div}(\nabla (D_iv))\,dx
\\
& = \sum_{i=1}^N\int_{\Gamma_i}(\vec{\nu}\cdot \nabla v) \Delta v\,dS-\sum_{j=1}^N \sum_{i=1}^N\int_{\Gamma_j} D_iv (\vec{\nu}\cdot \nabla(D_i v))\,dS+\int_\Omega \sum_{i,j=1}^N \left|D^2_{ij}v \right|^2
\,dx,
\end{split}
\]
where $\vec{\nu}$ is the outer normal to $\Gamma_i$. Since $\operatorname{supp} v \Subset \Omega$, the boundary integrals vanish.
\end{proof}

\begin{prop}
\label{pro:equiv-norm-1}
For every $v\in C_0^\infty(\Omega)$

\begin{equation}
\label{eq:equiv-norm-2}
c\|v\|^2_{H^2(\Omega)}\leq \|\Delta v\|^2_{2,\Omega}\leq \|v\|^2_{H^2(\Omega)}
\end{equation}
with an independent of $v$ constant $c$, and $\|\Delta v\|_{2,\Omega}$ is an equivalent norm of $C_0^\infty(\Omega)\cap H^2(\Omega)$.
\end{prop}

\begin{proof}
The second inequality is obvious. To prove the first one we represent $v\in C_0^\infty(\Omega)$ by the Fourier series in the basis $\{\psi_i\}$: $v=\sum_{i=1}^\infty v_i\psi_i$, $v_i=(v,\psi_i)_{2,\Omega}$,

\[
v^{(k)}=\sum_{i=1}^kv_i\psi_i\to v\quad \text{in $L^2(\Omega)$}.
\]
By inequality \eqref{eq:almost-Poincare} with $p_i=2$ we have $\|v\|_{2,\Omega}\leq C\|\nabla v\|_{2,\Omega}$ with a constant $C$ which does not depend on $v$. Since $v\in C_0^\infty(\Omega)$, then $\Delta v\in L^2(\Omega)$ and using the fact that $\{\psi_i\}$ are orthonormal, we get

\[
\|\Delta v\|^2_{2,\Omega}=\sum_{i=1}^{\infty}\lambda^2_i v_i^2<\infty.
\]
For every $k\in \mathbb{N}$

\begin{equation}
\label{eq:P-I}
\begin{split}
\|\nabla v^{(k)}\|_{2,\Omega}^{2}& = (\nabla v^{(k)},\nabla v^{(k)})_{2,\Omega}=-(\Delta v^{(k)},v^{(k)})_{2,\Omega}=\sum_{i=1}^{k}\lambda_iv_i^2
\\
&
\leq \left(\sum_{i=1}^k\lambda_i^2v_i^2\right)^\frac{1}{2}\left(\sum_{i=1}^k v_i^2\right)^\frac{1}{2}
\leq \|\Delta v\|_{2,\Omega}\|v^{(k)}\|_{2,\Omega}.
\end{split}
\end{equation}
Hence, $\|\nabla v^{(k)}\|_{2,\Omega}\leq C \|\Delta v\|_{2,\Omega}$ where the constant $C$ is from \eqref{eq:almost-Poincare}. Repeating these estimates for the function $\nabla (v^{(k)}-v^{(m)})$ we obtain

\[
\|\nabla (v^{(k)}-v^{(m)})\|_{2,\Omega}^2\leq 2\|\Delta v\|_{2,\Omega}\|v^{(k)}-v^{(m)}\|_{2,\Omega}\to 0\quad \text{as $k,m\to \infty$}.
\]
By the Cauchy inequality and due to monotonicity of the sequence $\{\lambda_i\}$

\[
\begin{split}
\|v^{(k)}\|^2_{2,\Omega} & =\sum_{i=1}^kv_i^2=\sum_{i=1}^k\frac{|v_i|}{\sqrt{\lambda_i}} \left(\sqrt{\lambda_i}|v_i|\right)\leq \dfrac{1}{\lambda_1}\|v^{(k)}\|_{2,\Omega}\|\nabla v^{(k)}\|_{2,\Omega}.
\end{split}
\]
Inequality \eqref{eq:P-I} entails the uniform estimate $\|\nabla v^{(k)}\|_{2,\Omega}\leq C\|\Delta v\|_{2,\Omega}$. It follows that $\nabla v^{(k)}\to \nabla v$ in $L^2(\Omega)$ and

\[
\|v\|_{2,\Omega}+\|\nabla v\|_{2,\Omega}\leq C'\|\nabla v\|_{2,\Omega}\leq C''\|\Delta v\|_{2,\Omega}
\]
with independent of $v$ constant $C''$. By Proposition \ref{pro:by-parts}, $\sum_{i,j=1}^{N}\|D^2_{ij}v\|^2_{2,\Omega}=\|\Delta v\|^2_{2,\Omega}$. Gathering these estimates we conclude that there is a constant $C>0$ such that for every $v\in C_0^\infty(\Omega)$

\[
C \sum_{0\leq |\alpha|\leq 2}\|D^\alpha v\|_{2,\Omega}^2\leq \|\Delta v\|_{2,\Omega}^2.
\]
\end{proof}

\begin{prop}
\label{pro:1}
Let $m\geq 2$ be an even number. There is a constant $C''>1$ such that for every $v\in C_0^{\infty}(\Omega)$

\[
\dfrac{1}{C''}\|\Delta^{\frac{m}{2}}v\|_{2,\Omega}\leq \|v\|_{H^m(\Omega)}\leq C''\|\Delta^{\frac{m}{2}}v\|_{2,\Omega}.
\]
\end{prop}

\begin{proof}
The first inequality immediately follows from the definition of the norm in $H^{m}(\Omega)$.
To prove the second one we argue by induction. For $m=2$ the required inequality coincides with \eqref{eq:equiv-norm-2}. Assume that $\|v\|_{H^{2k}(\Omega)}\leq C\|\Delta^{k}v\|_{2,\Omega}$ for some $k>1$. Set $g=\Delta v$. By the induction conjecture
\[
\|g\|_{H^{2k}(\Omega)}\leq C\|\Delta^k g\|_{2,\Omega}=C\|\Delta^{k+1}v\|_{2,\Omega}.
\]
Set $F_\alpha=D^\alpha v$, $|\alpha|=2k$. Since $F_\alpha \in C_0^\infty(\Omega)$, it follows from Proposition \ref{pro:equiv-norm-1} that
\[
\sum_{|\alpha|=2k}
\|D^{\alpha}g\|^2_{2,\Omega}=
\sum_{|\alpha|=2k}
\|\Delta F_\alpha\|^2_{2,\Omega}=\sum_{|\alpha|=2k}\sum_{i,j=1}^{N}\|D^{2}_{ij}F_\alpha\|^2_{2,\Omega}= \sum_{|\alpha|=2(k+1)}\|D^{\alpha}v\|_{2,\Omega}^2.
\]
Gathering the last two lines and using the induction conjecture we conclude that

\[
\|v\|_{H^{2(k+1)}(\Omega)}\leq C\|\Delta^{k+1}v\|_{2,\Omega}.
\]
\end{proof}

\begin{cor}
\label{cor:eigen-norm}
By virtue of \eqref{eq:mixed-eigenfunctions-even}, \eqref{eq:mixed-eigenfunctions-odd} the assertion of Proposition \ref{pro:1} is true for the eigenfunctions of problem \eqref{eq:eigen-Dir}: there is a constant $C=C(a,N,m)$ such that
\[
\|\psi_{\mathbf k}\|^2_{H^{2m}(\Omega)}=\sum_{0\leq |\alpha|\leq 2m}\|D^{\alpha}\psi_{\mathbf k}\|^2_{2,\Omega}\leq C\sum_{s=0}^m\|\Delta^s \psi_{\mathbf{k}}\|^{2}_{2,\Omega}\leq C\|\psi_{\mathbf k}\|^2_{H^{2m}(\Omega)}.
\]
\end{cor}

Given a function $f\in L^2(\Omega)$, let

\[
f^{(k)}=\sum_{i=1}^kf_i\psi_i,\qquad f_i=(f,\psi_i)_{2,\Omega},
\]
denote the partial sum of the Fourier series of $f$ in the basis $\{\psi_i\}$, $f^{(k)}\to f$ in $L^2(\Omega)$.

\begin{prop}
\label{epro:Fourier-m}
Let $m$ be a positive even integer. For every $f\in C_0^{\infty}(\Omega)$ and $\epsilon>0$ there is $l_0 \in \mathbb{N}$ such that

\[
\|f-f^{(l)}\|_{H^m(\Omega)}<\epsilon \quad \text{for all} \ l \geq l_0.
\]
\end{prop}

\begin{proof}
Set $m=2k$ and denote by $F_i$ and $f_i$ the Fourier coefficients of the functions $\Delta^{k}f$ and $f$ in the basis $\{\psi_i\}$ of $L^2(\Omega)$. The Fourier coefficients of $\Delta^{k}f$ are defined by

\[
\begin{split}
F_i & =(\Delta^{k}f,\psi_i)_{2,\Omega}=-(\nabla(\Delta^{k-1}f),\nabla \psi_i)_{2,\Omega}=-\lambda_i(\Delta^{k-1}f,\psi_i)_{2,\Omega}
\\
& = \ldots = (-1)^k\lambda^k_i (f,\psi_i)_{2,\Omega}=(-1)^k \lambda_i^kf_i,\qquad \forall i\in \mathbb{N}.
\end{split}
\]
Since $f\in C_0^{\infty}(\Omega)$, it follows from Proposition \ref{pro:1} that there exists a constant $C$ such that

\[
\frac{1}{C}\|f\|^2_{H^m(\Omega)}\leq \|\Delta^k f\|^2_{2,\Omega}=\sum_{i=1}^\infty F_i^2=\sum_{i=1}^\infty \lambda_i^{2k} f_i^2=\sum_{i=1}^\infty \lambda_i^{m} f_i^2<\infty.
\]
The convergence of this series means that the sequence of partial sums $\{f^{(s)}\}$ is a Cauchy sequence in $H^m(\Omega)$. Since $f^{(r)}$ are linear conbinations of $\{\psi_i\}_{i=1}^{s}$, it follows from Corollary \ref{cor:eigen-norm} that for $l<s$ and $m=2k$

\[
\|f^{(s)}-f^{(l)}\|_{H^{2k}(\Omega)}\leq C\sum_{i=l+1}^{s}\lambda_i^mf_i^2\to 0\quad \text{as $l\to \infty$}.
\]
Hence, $f^{(s)}\to f$ in $H^m(\Omega)$.
\end{proof}

\begin{lemma}
\label{le:density-stationary}
The system of eigenfunctions $\{\psi_i\}$ is dense in $W^{1,\vec{p}(\cdot)}_0(\Omega)$.
\end{lemma}

\begin{proof}
Take a function $v\in W_0^{1,\vec{p}(\cdot)}(\Omega)$ and fix an arbitrary $\epsilon>0$. By density of $C_0^\infty(\Omega)$ in $W^{1,\vec{p}(\cdot)}_0(\Omega)$, there exists $v_\epsilon\in C_0^{\infty}(\Omega)$ such that $\|v-v_\epsilon\|_{W_0^{1,\vec{p}(\cdot)}(\Omega)}<\epsilon$. By the definition of the norm in $W^{1,\vec{p}(\cdot)}_0(\Omega)$ and the generalized H\"older inequality, for every $w\in W^{1,\vec{p}(\cdot)}_0(\Omega)$

\[
\|w\|_{W_0^{1,\vec{p}(\cdot)}(\Omega)}=\sum_{i=1}^N\|D_iw\|_{p_i(\cdot),\Omega}\leq \sum_{i=1}^N
C_i(p_i^{\pm},|\Omega|)\|D_iw\|_{p^+,\Omega}\leq C\|\nabla w\|_{p^+,\Omega}.
\]
By the Sobolev embedding theorem, for every $w\in C_0^\infty(\Omega)$

\[
\|w\|_{W^{1,p^+}(\Omega)}\leq C \|\nabla w\|_{p^+,\Omega}\leq C'\|w\|_{H^{m}(\Omega)},
\]
where $p^+>\frac{2N}{N+2}$ and $m\geq 1+N\left(\frac{1}{2}-\frac{1}{p^+}\right)$ is an even integer. By Proposition \ref{epro:Fourier-m}, there is $k_0 \in \mathbb{N}$ such that $v_\epsilon^{(k)}=\sum_{i=1}^kv_{\epsilon i}\psi_i$ satisfies the inequality $\|v_\epsilon-v_{\epsilon}^{(k)}\|_{H^m(\Omega)}<\epsilon$, $\forall \ k \geq k_0$, and

\[
\|v-v_\epsilon^{(k)}\|_{W^{1,\vec{p}(\cdot)}_0(\Omega)}\leq \|v-v_\epsilon\|_{W^{1,\vec{p}(\cdot)}_0(\Omega)} +C'\|v_{\epsilon}-v_\epsilon^{(k)}\|_{H^m(\Omega)}<(1+C')\epsilon.
\]
\end{proof}

\subsection{Domains with smooth boundary}
If $\partial\Omega\in C^k$ with $k\geq 2$, we take for the basis of $L^2(\Omega)$ the set of eigenfunctions of the Dirichlet problem for the Laplace operator

\begin{equation}
\label{eq:eigen-smooth}
(\nabla \phi_i,\nabla \psi)_{2,\Omega}=\lambda_i(\nabla \phi_i,\psi)\qquad \forall \psi\in H^1_0(\Omega).
\end{equation}
It follows from the classical elliptic theory that $\psi_i\in H^{k}(\Omega)$. Define the closed subspace of $H^k(\Omega)$

\[
H_{\mathcal D}^k(\Omega)=\left\{u\in H^k(\Omega):\,\text{$\Delta^s u=0$ on $\partial\Omega\setminus \Gamma_0$},\,s=0,1,\ldots,\left[\frac{k-1}{2}\right]\right\},\qquad H_{\mathcal D}^0(\Omega)=L^2(\Omega).
\]
The relations
\[
[f,g]_{k}=\begin{cases}
(\Delta^{\frac{k}{2}}f,\Delta^{\frac{k}{2}}g)_{2,\Omega} & \text{if $k$ is even},
\\
(\Delta^{\frac{k-1}{2}}f,\Delta^{\frac{k-1}{2}}g)_{H^1(\Omega)} & \text{if $k$ is odd}
\end{cases}
\]
define an equivalent inner product on ${H}^{k}_{\mathcal{D}}(\Omega)$:
$[f,g]_k=\displaystyle\sum_{i=1}^\infty \lambda_i^k f_ig_i$,
where $f_i$, $g_i$ are the Fourier coefficients of $f$, $g$ in the basis $\{\phi_i\}$ of $L^2(\Omega)$. The corresponding equivalent norm of $H^k_{\mathcal{D}}(\Omega)$ is defined by $\|f\|^2_{{H}^{k}_{\mathcal{D}}(\Omega)}=[f,f]_k$. Let $f^{(m)}=\sum_{i=1}^{m}f_i\phi_i$ be the partial sums of the Fourier series of $f\in L^{2}(\Omega)$. The following assertion is well-known.
\begin{lemma}
\label{le:series}
Let $\partial\Omega\in C^k$, $k\geq 1$. A function $f$ can be represented by the Fourier series in the system $\{\phi_i\}$,
convergent in the norm of $H^k(\Omega)$, if and only if $f\in H^{k}_{\mathcal{D}}(\Omega)$. If $f\in H^{k}_{\mathcal{D}}(\Omega)$, then the series $\sum_{i=1}^{\infty}\lambda_i^kf_i^2$ is convergent, its sum is bounded by $C\|f\|_{H^k(\Omega)}$ with an independent of $f$ constant $C$, and $\|f^{(m)}-f\|_{H^k(\Omega)}\to 0$ as $m\to \infty$.
\end{lemma}

Let $u\in W_0^{1,\vec p(\cdot)}(\Omega)$ and $\epsilon>0$ be an arbitrary number. By Proposition \ref{pro:Fan-1} the set $C_0^\infty(\Omega)$ is dense in $W_0^{1,\vec p(\cdot)}(\Omega)$, therefore there exists $v_\epsilon\in C_0^{\infty}(\Omega)\subset H_{\mathcal{D}}^{k}(\Omega)$ such that $\|u-v_\epsilon\|_{W^{1,\vec p(\cdot)}_0(\Omega)}<\epsilon$. By Lemma \ref{le:series} one may find $m\in \mathbb{N}$ and $w_{\epsilon}\in \operatorname{span}\{\phi_1,\ldots,\psi_m\}$ such that $\|v_\epsilon-w_m\|_{H^k(\Omega)}<\epsilon$. Following the proof of Lemma \ref{le:density-stationary} we arrive at the following assertion.

\begin{lemma}
\label{le:density-smooth-boundary}
Set $\mathcal{P}_m=\operatorname{span}\{\phi_1,\ldots,\phi_m\}$. If $\partial\Omega\in C^k$ with $k\geq 1+N\left(\frac{1}{2}-\frac{1}{p^+}\right)$, and $p_i\in C_{\rm log}(\Omega)$, then $\cup_{m=1}^{\infty}\mathcal{P}_m$ is dense in $W^{1,\vec p(\cdot)}_0(\Omega)$.
\end{lemma}

\subsection{Spaces of functions depending on $z=(x,t)$}
Let $\vec p:\,\Omega\times (0,T)=Q_T\mapsto \mathbb{R}^N$ be a vector-valued function,

\[
\frac{2N}{N+2}<p_i(z)<p_h^\ast(z),\qquad p_i \in C^{0,1}(\overline{Q}_T).
\]
The space $\mathbb{W}(Q_T)$ defined in \eqref{eq:space-parab} is the closure of $C^\infty([0,T];C_0^\infty(\Omega))$ in the norm of $\mathbb{W}(Q_T)$. Let

\[
\mathcal{S}_m=\left\{u:\,u=\sum_{k=1}^m d_k(t)\psi_k(x),\,d_k\in C^{0,1}[0,T]\right\}.
\]
Then, $\bigcup_{m\geq 1}\mathcal{S}_m$ is dense in $\mathbb{W}(Q_T)$ (see \cite[Lemma 1.17]{ant-shm-book-2015}).

\subsection{The interpolation inequality}
Let $\vec p\in \mathbb{R}^N$ be a given constant vector such that
\begin{equation}
\label{eq:p-const-1}
\dfrac{2N}{N+2}<p^\wedge\leq p^\vee < p_h^\ast.
\end{equation}

\begin{prop}[Lemma 2.1, \cite{Acerbi-Fusco-1994}]
\label{pro:Acerbi-Fusco-1994}
Let $K_{\vec a}$ be a rectangular domain. If $\vec p$ satisfies condition \eqref{eq:p-const-1}, there exists a constant $C$ such that for every $u\in W^{1,(1,\vec p)}(K_{\vec a})$

\begin{equation}
\label{eq:A-F-1994}
\|u\|_{p_h^\ast,K_{\vec a}}\leq C\left(\sum_{i=1}^N\|D_iu\|_{p_i,K_{\vec a}}+\|u\|_{1,K_{\vec a}}\right).
\end{equation}
\end{prop}
Inequality \eqref{eq:A-F-1994} remains valid if the norm $\|u\|_{1,K_{\vec a}}$ is substituted by $\|u\|_{2,K_{\vec a}}$
It is well-known
(see, e.g., \cite[p.133]{Bresis-book}) that for every constant $s$, $q$ satisfying $2<s<q$

\begin{equation}
\label{eq:inter-1}
\|u\|_{s}\leq \|u\|_{q}^\theta\|u\|_{2}^{1-\theta}\quad \text{with}\quad \dfrac{1}{s}=\dfrac{\theta}{q}+\dfrac{1-\theta}{2}.
\end{equation}
Gathering \eqref{eq:inter-1} with \eqref{eq:A-F-1994} we obtain the following interpolation inequality.

\begin{lemma}
\label{le:interpol-global}
Let $\vec p$ satisfy \eqref{eq:p-const-1}. There exists a constant $C=C(N,a)$ such that for $u\in W^{1,(2,\vec{p})}(K_{\vec a})$

\begin{equation}
\label{eq:interpol-global}
\|u\|_{s,K_{\vec a}}\leq \|u\|^{\theta}_{p_h^\ast,K_{\vec a}}\|u\|_{2,K_{\vec a}}^{1-\theta} \leq C\sum_{i=1}^N\left(\|D_i u\|_{p_i,K_{\vec a}}
+\|u\|_{2,K_{\vec a}}\right)^{\theta}\|u\|_{2,K_{\vec a}}^{1-\theta},
\end{equation}
for every $2<s< p_h^\ast$ and

\[
\theta=
\dfrac{\frac{1}{2}-\frac{1}{s}}{\frac{N+2}{2N}-\frac{1}{p_h}}\in (0,1).
\]
\end{lemma}

The interpolation inequality \eqref{eq:interpol-global} can be adapted to the case of variable exponents $p_i$. Given a vector $\vec p(x)$ defined on $K_{\vec a}$, set

\begin{equation}
\label{eq:p_h-}
p_i^-=\min_{\overline{K}_{\vec a}}p_i(x), \qquad p_i^+=\max_{\overline{K}_{{\vec a}}}p_i(x),\qquad p_h^-=\dfrac{N}{\sum_{i=1}^N\frac{1}{p_i^-}}.
\end{equation}

\begin{lemma}
\label{le:interpol-var}
Assume that $p_i \in C^0(\overline{K}_{\vec a})$ for all $i=\overline{1, N}$, and
\[
\frac{2N}{N+2}
<p_i^-\leq p_i(x)\leq p_i^+ < (p_h^{-})^\ast=\begin{cases} \frac{Np^-_h}{N-p_h^-} & \text{if $p_h^-<N$}
\\
\text{any finite number} & \text{if $p_h^-\geq N$}.
\end{cases}
\]
Then there exists a constant $C=C(N,\vec a)$ such that for every $u\in W^{1,(2,\vec{p}(\cdot))}(K_{\vec a})$

\begin{equation}
\label{eq:interpol-global-var}
\begin{split}
\|u\|_{s(\cdot),K_{\vec a}} & \leq C\|u\|_{s^+,K_{\vec a}}\leq \|u\|^{\theta}_{p_h^\ast,K_{\vec a}}\|u\|_{2,K_{\vec a}}^{1-\theta}
\\
&
\leq C\sum_{i=1}^N\left(\|D_i u\|_{p^-_i,K_{\vec a}}+ \|u\|_{2,K_{\vec a}} \right)^{\theta}\|u\|_{2,K_{\vec a}}^{1-\theta},
\end{split}
\end{equation}
where $s\in C^0(\overline{K}_{\vec a})$,

\begin{equation}
\label{eq:r-crit}
2<s^+=\max_{\overline{K}_{\vec a}}s(x)< (p_h^{-})^\ast,\qquad
\theta=\dfrac{\frac{1}{2}-\frac{1}{s^+}}{\frac{N+2}{2N}-\frac{1}{p_h^-}}\in (0,1).
\end{equation}
\end{lemma}

\begin{cor}
Let the conditions of Lemma \ref{le:interpol-var} be fulfilled and $\|u\|_{2,\Omega}=M_0$ be a known constant. Then

\begin{equation}
\label{eq:global-conv}
\|u\|_{s^+,K_{\vec a}}\leq C\sum_{i=1}^N\|D_i u\|_{p^-_i,K_{\vec a}}^{\theta}+ C'
\end{equation}
with constants $C$, $C'$ depending on $a$, $N$, and $M_0$.
\end{cor}

\section{Regularization. The approximate problems}
\label{sec:regularization}
The solution of the problem \eqref{eq:main} is obtained as the limit of the family $\{u_\epsilon\}_{\epsilon >0}$ of solutions to the regularized nondegenerate problems

\begin{equation}
\label{eq:main-reg}
\begin{split}
& u_t-\sum_{j=1}^{N}D_j\left((\epsilon^2+|D_ju|^2)^{\frac{p_j(z)-2}{2}}D_ju\right)=f\quad \text{in $Q_T$},
\\
& \text{$u=0$ on $\partial\Omega\times (0,T)$},
\qquad
\text{$u(x,0)=u_0(x)$ in $\Omega$},\quad \epsilon \in (0,1).
\end{split}
\end{equation}
The solution of the problem \eqref{eq:main-reg} with $\epsilon\in (0,1)$ is a function satisfying the following conditions

\begin{itemize}
\item[(i)] $u_\epsilon \in C^0([0,T];L^2(\Omega))\cap \mathbb{W}(Q_T)$, $u_{\epsilon t}\in L^{2}(Q_T)$,

\item[(ii)] for every $\phi\in \mathbb{W}(Q_T)$

\begin{equation}
\label{eq:def-main-reg}
\int_{Q_T}\left(u_{\epsilon t}\phi+\sum_{i=1}^{N}(\epsilon^2+|D_iu_\epsilon|^{2})^{\frac{p_i(z)-2}{2}}D_iu_\epsilon D_i\phi\right)\,dz=\int_{Q_T}f\phi\,dz,
\end{equation}

\item[(iii)] for every $\phi\in L^2(\Omega)$ $(u_\epsilon(\cdot,t)-u_0,\phi)_{2,\Omega}\to 0$ as $t\to 0^+$.
\end{itemize}

\begin{thm}
\label{th:existence-reg}
Let $\Omega=K_{\vec a}$ and $\vec p(z)$ satisfy the conditions of Theorem \ref{th:existence}. Then for every $\epsilon \in (0,1)$ and every $u_0\in W^{1,2}_0(\Omega)\cap W^{1,\vec p(\cdot,0)}_0(\Omega)$ and $f\in L^2(0,T; W^{1,2}_0(\Omega))$, problem \eqref{eq:main-reg} has a unique weak solution

\[
u_\epsilon \in \mathbb{W}(Q_T)\cap L^{\infty}(0,T;W^{1,2}_0(\Omega)\cap W^{1,\vec p(\cdot)}_0(\Omega)),
\]
with

\begin{equation}
\label{eq:est-epsilon}
\begin{split}
\displaystyle \|u_{\epsilon t}\|_{2,Q_T} & +\operatorname{ess}\sup\limits_{(0,T)}\|u_\epsilon\|_{W_0^{1,2}(\Omega)}+ \operatorname{ess}\sup\limits_{(0,T)}\|u_\epsilon\|_{W_0^{1,\vec p(\cdot, t)}(\Omega)}
\\
& \leq C\left(1+\|f\|_{L^2(0,T;W^{1,2}(\Omega))}+ \|u_0\|_{W_0^{1,2}(\Omega)}+ \|u_0\|_{W_0^{1,\vec p(x,0)}(\Omega)}\right)
\end{split}
\end{equation}
with an independent of $\epsilon$ constant $C$. Moreover,

\begin{equation}
\label{eq:high-reg}
\sum_{i=1}^{N}\int_{Q_T}|D_iu_\epsilon|^{p_i(z)+r}\,dz\leq C\qquad \text{for every $\displaystyle r\in \left(0, \frac{2N(1-\mu)+4}{N+2}\right)$}
\end{equation}
with a constant $C$ independent of $\epsilon$, and the constant $\mu$ defined in \eqref{eq:p-i-2}.
\end{thm}

A solution of problem \eqref{eq:main-reg} is constructed as the limit of the sequence of finite-dimensional Galerkin's approximations in the basis $\{\psi_i\}$. Let $u^{(m)}=\sum_{i=1}^m c_{m,i}(t)\psi_i(x)$ where the coefficients $\vec c_m=(c_{m,1},\ldots,c_{m,m})$
are defined from the system of ordinary differential equations

\begin{equation}
\label{eq:ODE-m}
\begin{split}
& c'_{m,i}(t)=-\sum_{j=1}^N\left(\mathcal{F}^{(\epsilon)}_j(z,D_j u_\epsilon^{(m)}),D_j\psi_i\right)_{2,\Omega}+(f(\cdot, t),\psi_i)_{2,\Omega},
\\
& c_{m,i}(0)=(u_0^{(m)},\psi_i)_{2,\Omega}, \quad i=\overline{1,m},
\\
& \mathcal{F}^{(\epsilon)}_j(z,D_j u_\epsilon^{(m)}):=(\epsilon^2+|D_ju_\epsilon^{(m)}|^2)^{\frac{p_j(z)-2}{2}}D_ju_\epsilon^{(m)}.
\end{split}
\end{equation}
By Lemma \ref{le:density-stationary}, the sequence $\{c_{m,i}(0)\}$ can be chosen so that

\[
u_0^{(m)}(x)=\sum_{i=1}^{m}c_{m,i}(0)\psi_i\to u_0\quad \text{in $W_0^{1,\vec q(\cdot)},\quad q_i(x)=\max\{p_i(x,0),2\}$}.
\]
By the Carath\'eodory existence theorem, for every $u_0\in W^{1,\vec q(x)}(\Omega)$, $f\in L^2(0,T;W_0^{1,2}(\Omega))$ and $m\in \mathbb{N}$ problem \eqref{eq:ODE-m} has a solution $\vec c_m = (c_{m,1}, c_{m,2}, \dots, c_{m,m})$ on an interval $[0,T_m]$. The possibility of continuation of each of $c_{m,i}$ to the interval $[0,T]$ will follow from the uniform a priori estimates on $u_\epsilon^{(m)}$ derived in the next section.

The global second-order differentiability for the solution to \eqref{eq:main-reg} is given in the following theorem.

\begin{thm}
\label{th:high-diff-reg}
Under the conditions of Theorem \ref{th:existence-reg}

\[
(\epsilon^2+|D_iu_\epsilon|^2)^{\frac{p_i(z)-2}{4}}D_iu_\epsilon\in W^{1,2}(Q_T),\qquad i=\overline{1,N}.
\]
\end{thm}

\section{A priori estimates}
\label{sec:estimates}
Fix a number $m\in \mathbb{N}$ and consider the function $u_\epsilon^{(m)}$. For the sake of presentation, when deriving the a priori estimates for the solutions of the regularized problem \eqref{eq:main-reg} we omit the indexes $\epsilon$ and $m$ and write $u\equiv u^{(m)}_\epsilon$.  Multiplying the $i$th equation of \eqref{eq:ODE-m} by $c_{m,i}$ and summing up for $i=\overline{1,m}$, we arrive at the relation 

\begin{equation}
\label{eq:1-est}
\begin{split}
\dfrac{1}{2}\dfrac{d}{dt}\|u(\cdot, t)\|_{2,\Omega}^2 & + \int_{\Omega}\sum_{i=1}^{N}(\epsilon^2+|D_iu|^2)^{\frac{p_i(z)-2}{2}}|D_iu|^2\,dx =\int_{\Omega}f u\,dx
\leq \frac{1}{2}\|f(\cdot, t)\|_{2,\Omega}^2+\frac{1}{2}\|u(\cdot, t)\|_{2,\Omega}^2.
\end{split}
\end{equation}

\begin{lemma}
\label{le:a-priori-1}
The approximate solutions $u\equiv u^{(m)}_\epsilon$ satisfy the uniform estimate

\begin{equation}
\label{eq:1st-est}
\sup_{(0,T)}\|u(t)\|_{2,\Omega}^2 + C\sum_{i=1}^N\int_{Q_T}(\epsilon^2+|D_i u|^2)^{\frac{p_i(z)-2}{2}}|D_iu|^2\,dz\leq C'\left(\|f\|_{2,Q_T}^2 +\|u_0\|_{2,\Omega}^{2}\right)
\end{equation}
with independent of $m$ and $\epsilon$ constants $C=C(T)$, $C'=C'(T)$.
\end{lemma}

\begin{proof}
Estimate \eqref{eq:1st-est} follows after multiplication of \eqref{eq:1-est} by ${\rm e}^{-t}$ and integration in $t$.
\end{proof}

Multiplying the $i$th equation of \eqref{eq:ODE-m} by $c'_{m,i}(t)$, summing the results over $i=\overline{1,m}$, integrating in $t$ and taking into account the identities

\[
\begin{split}
& (\epsilon^2+|D_iu|^2)^{\frac{p_i-2}{2}}D_iuD_iu_t = \frac{1}{2}(\epsilon^2+|D_iu|^2)^{\frac{p_i-2}{2}}\left(|D_iu|^2\right)_t
\\
& \qquad = \dfrac{\partial}{\partial t}\left(\frac{1}{p_i}(\epsilon^2+|D_iu|^2)^{\frac{p_i}{2}}\right) +\frac{p_{it}}{p_i^2}(\epsilon^2+|D_iu|^2)^{\frac{p_i}{2}} -\dfrac{p_{it}}{p_i}(\epsilon^2+|D_iu|^2)^{\frac{p_i}{2}}\ln (\epsilon^2+|D_iu|^2)
\end{split}
\]
we arrive at the inequality
\[
\begin{split}
\|u_t\|_{2,Q_T}^2 & +\sum_{i=1}^N\int_{\Omega}\frac{1}{p_i}(\epsilon^2+|D_iu|^2)^{\frac{p_i(z)}{2}} \,dx
\leq \frac{1}{2}\|u_t\|_{2,Q_T}^2 +\frac{1}{2}\|f\|_{2,Q_T}^2
\\
& + \sum_{i=1}^N\int_{\Omega}\frac{1}{p_i(x,0)}(\epsilon^2+|D_iu_0|^2)^{\frac{p_i(x,0)}{2}} \,dx
\\
& -\sum_{i=1}^N \int_{Q_T}\left(\frac{p_{it}}{p_i^2}(\epsilon^2+|D_iu|^2)^{\frac{p_i(z)}{2}} +\dfrac{p_{it}}{p_i}(\epsilon^2+|D_iu|^2)^{\frac{p_i(z)}{2}}\ln (\epsilon^2+|D_iu|^2)\right)\,dz.
\end{split}
\]
The estimate on the first integral in the last line follows from \eqref{eq:1st-est}. The second integral is bounded by

\[
C_\mu\sum_{i=1}^N\int_{Q_T}\left(1+|D_iu|^2)^{\frac{p_i(z)+\rho}{2}}\right)\,dz
\]
with any constant $\rho>0$ by virtue of the following elementary inequalities:  for every $\rho>0$, there is a constant $C_\rho$ such that

\begin{equation}
\label{eq:elem-1}
\ln^2 s\leq C_{\rho}\begin{cases} s^{\rho} & \text{if $s\geq 1$},
\\
s^{-\rho} & \text{if $s\in (0,1)$},
\end{cases}
\end{equation}
and

\[
(\epsilon^2+|\vec \xi|^2)^{\frac{p_i+\rho}{2}}\leq \begin{cases}
(\sqrt{2}\epsilon)^{p_i+\rho}\leq 1 & \text{if $|\vec \xi|< \epsilon$},
\\
(\sqrt{2}|\vec \xi|)^{p_i+\rho} & \text{if $|\vec \xi|\geq \epsilon$}
\end{cases}
\leq C\left(1+|\vec \xi|^{p_i+\rho}\right),\quad \epsilon\in (0,1).
\]

\begin{lemma}
\label{le:a-priori-2}
If $|p_{it}|\leq L$ a.e. in $Q_T$, then for every constant $\rho\in (0,1)$ the functions $u\equiv u^{(m)}_\epsilon$ satisfy the estimates

\begin{equation}
\label{eq:2nd-est}
\begin{split}
\|u_t\|_{2,Q_T}^2 & +\sum_{i=1}^N\sup_{(0,T)}\int_{\Omega}\frac{1}{p_i}(\epsilon^2+|D_iu|^2)^{\frac{p_i(z)}{2}} \,dx
\\
&
\leq C\left(1+\|f\|_{2,Q_T}^2
+ \sum_{i=1}^N\int_{\Omega}|D_iu_0|^{p_i(x,0)}\,dx\right)
+C'\int_{Q_T}|D_iu|^{p_i(z)+\rho}\,dz
\end{split}
\end{equation}
with constants $C,C'$ depending on $\rho$ but independent of $m$ and $\epsilon$.
\end{lemma}

Let us multiply the nonlinear flux term in the $k$th equation of \eqref{eq:ODE-m} by $\lambda_k c_{m,k}(t)$, and sum up the results for $k=\overline{1,m}$. Recall the convention to denote $u\equiv u_\epsilon^{(m)}$. For every $i=\overline{1,N}$

\begin{equation}
\label{eq:3rd-est-pre}
\begin{split}
& -\left(\mathcal{F}^{(\epsilon)}_i(z, D_i u),\sum_{k=1}^m\lambda_k c_{m,k}(t) D_i\psi_k\right)_{2,\Omega}
=\int_\Omega D_i((\epsilon^2+|D_{i}u|^2)^{\frac{p_i(z)-2}{2}}D_iu)\Delta u\,dx
\\
&
=\int_{\partial \Omega}(\epsilon^2+|D_{i}u|^2)^{\frac{p_i(z)-2}{2}}D_iu \cos(\vec{\nu},x_i)\Delta u\,dS - \int_{\Omega}(\epsilon^2+|D_{i}u|^2)^{\frac{p_i(z)-2}{2}}D_iu \sum_{j=1}^N D_j(D^2_{ij}u)\,dx
\\
& = \int_{ \partial \Omega}(\epsilon^2+|D_{i}u|^2)^{\frac{p_i(z)-2}{2}}D_iu \left( \cos(\vec{\nu},x_i)\Delta u-\sum_{j=1}^N \cos(\vec{\nu},x_j)D^2_{ij}u\right)\,dS
\\
&
\qquad +\sum_{j=1}^N\int_{\Omega}D_j\left((\epsilon^2+|D_{i}u|^2)^{\frac{p_i(z)-2}{2}}D_iu\right) D^2_{ij}u\,dx,
\end{split}
\end{equation}
where $\vec{\nu}$ is the exterior normal to $\partial \Omega$. In the first line of \eqref{eq:3rd-est-pre} we integrated by parts and used the fact that $\psi_k=0$ on $\partial\Omega$. By splitting the integrals over $\partial \Omega$ into $\Gamma_i$ and using the fact that for the rectangular domain $\cos(\vec{\nu},x_j)|_{\Gamma_i}=0$ if $i\not =j$, we obtain
\[
\int_\Omega D_i((\epsilon^2+|D_{i}u|^2)^{\frac{p_i(z)-2}{2}}D_iu)\Delta u\,dx= \sum_{j=1}^N\int_{\Omega}D_j((\epsilon^2+|D_{i}u|^2)^{\frac{p_i(z)-2}{2}}D_iu) D^2_{ij}u\,dx:=\mathcal{I}_i.
\]
The straightforward computation shows that

\[
\begin{split}
\mathcal{I}_i & =\sum_{j=1}^{N}\int_{\Omega}(p_i-1)(\epsilon^2+|D_{i}u|^2)^{\frac{p_i(z)-2}{2}}\left(D^2_{ij}u\right)^2\,dx
\\
& + \sum_{j=1}^N \int_\Omega (\epsilon^2+|D_{i}u|^2)^{\frac{p_i(z)-2}{2}}D_iu D_jp_i \ln (\epsilon^2+|D_iu|^2) D^2_{ij}u \,dx\equiv \mathcal{J}_1+\mathcal{J}_2.
\end{split}
\]
The second term is bounded by Young's inequality: for every $\delta>0$

\[
\begin{split}
|\mathcal{J}_2| & \leq \delta \sum_{j=1}^N \int_{\Omega}(\epsilon^2+|D_{i}u|^2)^{\frac{p_i(z)-2}{2}}\left(D^2_{ij}u\right)^2\,dx
+  C_\delta\sum_{j=1}^N \int_{\Omega}(\epsilon^2+|D_{i}u|^2)^{\frac{p_i(z)}{2}}
\ln^2 (\epsilon^2+|D_{i}u|^2)\,dx.
\end{split}
\]
Applying \eqref{eq:elem-1} we find that

\[
\begin{split}
\mathcal{J}_2 & \leq \delta \sum_{j=1}^N \int_{\Omega}(\epsilon^2+|D_{i}u|^2)^{\frac{p_i(z)-2}{2}}\left(D^2_{ij}u\right)^2\,dx
+ C \left(1+\int_{\Omega}|D_iu|^{p_i(z)+\rho}\,dx\right).
\end{split}
\]
Choose $\delta$ so small that $\min_ip_i^->1+\delta$. For every $\rho\in (0,1)$, we obtain the inequality

\begin{equation}
\label{eq:inter-new}
\begin{split}
\dfrac{1}{2}\dfrac{d}{dt}\left(\|\nabla u\|^2_{2,\Omega}\right) & + \sum_{i,j=1}^N\int_{\Omega}(p_i-1-\delta)(\epsilon^2+|D_iu|^2)^{\frac{p_i(z)-2}{2}} \left(D^2_{ij}u\right)^2\,dx
\\
&
\leq C\left(1+\sum_{i=1}^N\int_{\Omega}|D_iu|^{p_i(z)+\rho}\,dx\right) + \int_\Omega \nabla f\cdot \nabla u\,dx.
\end{split}
\end{equation}
The last term is estimated by

\[
(\nabla f,\nabla u)_{2,\Omega}\leq \frac{1}{2}\|\nabla f\|^2_{2,\Omega} + \frac{1}{2}\|\nabla u\|_{2,\Omega}^{2}.
\]
Arguing as in the proof of Lemma \ref{le:a-priori-1} we obtain

\begin{lemma}
\label{le:a-priori-3}
If $|\nabla p_i|\leq L$ a.e. in $Q_T$, then the functions $u\equiv u_\epsilon^{(m)}$ satisfy the estimates

\begin{equation}
\label{eq:3rd-est}
\begin{split}
\sup_{(0,T)}\|\nabla u(t)\|^2_{2,\Omega} & + \sum_{i,j=1}^N\int_{Q_T}(\epsilon^2+|D_iu|^2)^{\frac{p_i(z)-2}{2}} \left(D^2_{ij}u\right)^2\,dz
\\
&
\leq C\left(1+\sum_{i=1}^N\int_{Q_T}|D_iu|^{p_i(z)+\rho}\,dz + \|\nabla u_0\|_{2,\Omega}^2+\| \nabla f\|_{2,Q_T}^{2}\right)
\end{split}
\end{equation}
with any $\rho\in (0,1)$ and a constant $C$ independent of $m$, $\epsilon$.
\end{lemma}

\begin{remark}
\label{rem:alternative-est}
An analogue of estimate \eqref{eq:3rd-est} holds true if $|p_t|+|\nabla p_i|\leq L$ a.e. in $Q_T$ and $f\in L^2(Q_T)\cap W^{1,\vec q(\cdot)'}_0(\Omega)$ with $q_i(z)=\max\{p_i(z),2\}$. In this case we apply the Young inequality to obtain

\[
\begin{split}
(\nabla f,\nabla u)_{2,\Omega} & \leq \sum_{i=1}^N\int_{\Omega}\dfrac{1}{p'_i(z)}|D_if|^{q'_i(z)}\,dx +
\sum_{i=1}^N\int_{\Omega}\dfrac{1}{q_i(z)}|D_iu|^{q_i(z)}\,dx
\\
& \leq \sum_{i=1}^N \int_{\Omega}|D_if|^{q'_i(z)}\,dx + C\left(\|\nabla u\|_{2,\Omega}^2+ \sum_{i=1}^N\int_{\Omega}|D_iu|^{p_i(z)}\,dx\right)
\end{split}
\]
and use Lemma \ref{le:a-priori-2} to estimate the last term.
\end{remark}

\begin{remark}
\label{rem:integrability-1}
Let us assume that $x_i$ is the direction of fast diffusion: there exists $\sigma>0$ such that $p_i(z)+\sigma\leq 2$ in $\overline{Q}_T$. Then the integral of $|D_i u|^{p_i(z)+\rho}$ on the right-hand side of \eqref{eq:inter-new} can be estimated by Young's inequality, provided $0<\rho <\sigma $:

\[
\int_{\Omega}|D_iu|^{p_i(z)+\rho}\,dx\leq C\left(1+\int_\Omega|\nabla u|^2\,dx\right).
\]
Multiplying \eqref{eq:inter-new} by ${\rm e}^{-Ct}$ we absorb these terms in the derivative of $\|\nabla u\|_{2,\Omega}^{2}{\rm e}^{-Ct}$ on the left-hand side and integrate the result in $t$. It follows that the right-hand side of \eqref{eq:3rd-est} does not include the integrals of $|D_iu|^{p_i(z)+\rho}$ corresponding to the directions of fast diffusion.
The integrals of $|D_iu|^{p_i(z)+\rho}$ with $\max p_i(z)\geq 2$ require special estimating.
\end{remark}

\section{Higher integrability of the gradients}
\label{sec:higher-integr}
Let us fix the index $m\in \mathbb{N}$ and consider the function $u\equiv u_\epsilon^{(m)}$. Integrating by parts we find that for every constant $r>0$

\begin{equation}
\label{eq:start}
\begin{split}
\int_{\Omega} & (\epsilon^2+|D_iu|^2)^{\frac{p_i+r-2}{2}} (D_iu)^2\,dx= \int_{\Omega} (\epsilon^2+|D_iu|^2)^{\frac{p_i+r-2}{2}}   D_iuD_iu\,dx \\
&
=\int_{\partial \Omega}u \cos(\vec{\nu},x_i)D_iu (\epsilon^2+|D_iu|^2)^{\frac{p_i+r-2}{2}}\,dS - \int_{\Omega} u (\epsilon^2+|D_iu|^2)^{\frac{p_i+r-2}{2}}  D_{ii}^2u\,dx\\
& \quad - \frac{1}{2} \int_\Omega uD_iu(\epsilon^2+|D_iu|^2)^{\frac{p_i+r-2}{2}}\ln ((\epsilon^2+|D_iu|^2) D_ip_i\,dx
\\
& \quad - \int_{\Omega} u (p_i+r-2)(\epsilon^2+|D_iu|^2)^{\frac{p_i+r-4}{2}} |D_i u|^2 D_{ii}^2u\,dx
\\
& \leq C \int_{\Omega} |u| (\epsilon^2+|D_iu|^2)^{\frac{p_i+r-2}{2}} |D_iu| \ln ((\epsilon^2+|D_iu|^2) \,dx + C_0 \int_{\Omega} |u| (\epsilon^2+|D_iu|^2)^{\frac{p_i+r-2}{2}}  D_{ii}^2u\,dx \\
& \equiv \mathcal{I}_1 + \mathcal{I}_2.
\end{split}
\end{equation}
The first integral over $\partial \Omega$ on the right-hand side of \eqref{eq:start} vanishes because $u=0$ on $\Gamma_i$. The integrals $\mathcal{I}_i$ are estimated separately.

\medskip

\noindent \textbf{Estimate for $\mathcal{I}_1$.} For the sake of simplicity of notation, we assume first that the exponents $p_i$ are independent of $t$. For every $\nu\in (0,1)$ and $\lambda>0$
\[
\begin{split}
\mathcal{I}_1 & \leq C_1\int_\Omega |u|(\epsilon^2+|D_iu|^2)^{\frac{p_i+r-1+\nu}{2}}\,dx + C_2(1+\|u\|_{2,\Omega}^2)
\\
& \leq C+\lambda \int_{\Omega} (\epsilon^2+|D_iu|^2)^{\frac{p_i+r-2}{2}}|D_iu|^2\,dx + C_\lambda \int_{\Omega} |u|^{\frac{p_i+r}{1-\nu}}\,dx,
\end{split}
\]
which allows one to choose $\lambda$ so small that the first term can be absorbed in the left-hand side of \eqref{eq:start}. To study the second term, we cover $\Omega=K_{\vec a}$ by a finite number of cubes $K_{\vec a,b_k}$ with the edge length $b_k \leq \beta$ such that
\[
 K_{\vec a}= \bigcup_{k=1}^\ell K_{\vec a,b_k}.
\]
The cubes from the cover $\{K_{\vec a,b_k}\}_{k=1}^{\ell}$ may overlap. It is sufficient to derive the needed estimates for each of $K_{\vec a,b_k}$ and then sum up the results. The number of cubes $K_{\vec a,b_k}$ in the chosen cover depends on $a_i$ and on the modules of continuity of $p_i(x)$. For the sake of simplicity of notation we will denote $K_{\vec a,b_k}=K_{b_k}$. Take a cube $K_{b_k}$, a vector $\vec r=(r,r,\ldots,r)\in \mathbb{R}^N$, $r\in (0,1)$, and set

\begin{equation}
\label{eq:p_h--1}
\begin{split}
& p_i^-=\min_{\overline{K}_{b_k}} p_i(x), \quad p_i^+=\max_{\overline{K}_{b_k}}p_i(x),
\\
& \vec q=\vec p+\vec r, \qquad \vec q^-=(q_1^-,\ldots,q_N^-),
\\
&
q_{h}^-=\dfrac{N}{\sum_{i=1}^N\frac{1}{q_i^-}} \quad \text{ - the harmonic mean of $\vec q^-$}.
\end{split}
\end{equation}
To estimate the integrals of $|u|^{\frac{p_j+r}{1-\nu}}$, $j=\overline{1,N}$, we want to apply \eqref{eq:global-conv} in the cube $K_{b_k}$ to a function $u\in W^{1,(2,\vec q^-)}(K_{b_k})$ with the exponent

\[
s_j=\dfrac{p_j^++r}{1-\nu}\equiv \frac{q_j^+}{1-\nu},\quad j=\overline{1,N},\quad \nu\in (0,1).
\]
If the parameters satisfy the conditions

\begin{equation}
\label{eq:cond-++}
\begin{split}
{(a)} & \qquad \frac{2N}{N+2}<q^-_i\leq q_i^+ < (q_h^-)^\ast=\begin{cases}
\frac{N q_{h}^-}{N-q_h^-} & \text{if $N>q_h^-$},
\\
\text{any number from $[1,\infty)$} & \text{if $N\leq q_h^-$},
\end{cases}
\\
(b) & \qquad s_j<(q_h^-)^\ast,
\\
(c) & \qquad \frac{s_j\theta_j}{q_i^-} < 1, \qquad \theta_j=\dfrac{\frac{1}{2}-\frac{1}{s_j}}{\frac{N+2}{2N}-\frac{1}{q_h^-}}\in (0,1),
\end{split}
\end{equation}
then for every $j=\overline{1,N}$

\[
\int_{K_{b_k}}|u|^{s_j}\,dx\leq C\sum_{i=1}^N \left(\int_{K_{b_k}}|D_i u|^{q_i^-}\,dx\right)^{\frac{s_j\theta_j}{q_i^-}}+C'
\]
with constants $C$, $C'$ depending on $p^\pm_j$, $a$, $N$ and $\|u\|_{2,K_{b_k}}$. The extra condition \eqref{eq:cond-++} (c) will provide the possibility to extend this estimate to functions defined on the cylinder with the base $K_{b_k}$. Condition \eqref{eq:cond-++} (c) can be written as follows: for all $i,j=\overline{1,N}$

\begin{equation}
\label{eq:prelim-1}
\dfrac{s_j\theta_j}{q_i^-} < 1\quad \Leftrightarrow \quad \frac{s_j}{2}-1 < q_i^{-}\left(\frac{N+2}{2N}-\frac{1}{q_h^-}\right) \quad \Leftrightarrow \quad s_j = \frac{q_j^+}{1-\nu} < \frac{N+2}{N} q_i^- + 2\left(1-\frac{q_i^-}{q_h^-}\right).
\end{equation}
By continuity of $\vec q(x)$, \eqref{eq:prelim-1} is true for a sufficiently small cube $K_{b_k}$ and small $\nu$, provided that the following strict inequality is fulfilled:

\begin{equation}
\label{eq:intermediate-j}
q_j^+< \frac{N+2}{N} q_i^- + 2\left(1-\frac{q_i^-}{q_h^-}\right)\equiv 2+ 2q_i^-\left(\frac{N+2}{2N}-\frac{1}{q_h^-}\right),\qquad i,j=\overline{1,N}.
\end{equation}
Note that the indexes on the right and the left-hand sides of \eqref{eq:intermediate-j} are not related. Accept the notation

\begin{equation}
\label{eq:Q}
\sigma^+=\max_j q_j^+,\qquad \sigma^-=\min_j q_j^-.
\end{equation}
Since $\frac{1}{q_h^-} < \frac{N+2}{2N}$, inequality \eqref{eq:intermediate-j} is fulfilled for all $i,j$ if

\begin{equation}
\label{eq:oscillation-q}
0\leq \sigma^+-\sigma^-<2 +2\sigma^-\left(\frac{N+2}{2N}-\frac{1}{2}-\frac{1}{q_h^-}\right)=2+2\sigma^-\left(\frac{1}{N}-\frac{1}{q_h^-}\right).
\end{equation}

\begin{prop}
\label{pro:choice-1}
Let $\max_j\|\nabla p_j\|_{\infty,\Omega}= L$ and

\begin{equation}
\label{eq:oscillation}
\mu\equiv \sup_\Omega \dfrac{p^\vee(x)}{p^\wedge(x)}<1+\frac{2}{N}.
\end{equation}
Then condition \eqref{eq:oscillation-q} is fulfilled in every cube $K_{b_k}$ with the parameter

\[
\nu \leq 1-(\mu+\gamma)\frac{N}{N+2},
\]
if $b_k \leq \beta =\beta(a, L,N,p^\pm_j)$ for all $k=1,2, \ldots, \ell$ with $\beta$ so small that

\begin{equation}
\label{eq:gamma}
\gamma=2 \beta L\sqrt{N}\dfrac{(N+2)^2}{4N^2}(\max_jp_j^++\min_jp_j^-+2)<1+\dfrac{2}{N}-\mu.
\end{equation}
\end{prop}

\begin{proof}
Take an arbitrary cube $K_{b_k}$ and assume that at least one of $p_j(x)$ is nonconstant. The case of the constant $\vec p$ is a simple corollary. By convention we use notation \eqref{eq:p_h--1}. Since the value $p_j^+$ is attained at some point $x_0\in \overline{K}_{b_k}$, by the Lagrange mean value theorem

\[
p_j(x)=p_j^{+}+\int_{0}^{1}\nabla p_j(sx+(1-s)x_0)\,ds \cdot (x-x_0)\geq p_j^+-\delta,\qquad \delta=2 \beta \sqrt{N}L,
\]
and $p_j(x)\leq p_j^-+\delta$. Assume $\beta$ is so small that $2 \beta L\sqrt{N}<\max_jp_j^+-\min_jp_j^-$. For every $x\in K_{b_k}$

\begin{equation}
\label{eq:mu-new}
\mu\geq \dfrac{p^\vee(x)}{p^\wedge(x)}
\geq \dfrac{\max\{p^+_j+r-\delta,\,j=\overline{1,N}\}}{\min\{p_j^-+r+\delta,\,j=\overline{1,N}\}}\geq \dfrac{\sigma^+-\delta}{\sigma^-+\delta}\geq -\gamma +\dfrac{\sigma^+}{\sigma^-}
\end{equation}
with $\gamma>\delta(\sigma^++\sigma^-)(\sigma^-)^{-2}$. The last inequality follows from the definition \eqref{eq:gamma} of $\gamma$:

\[
\gamma=2 \beta L\sqrt{N}\dfrac{(N+2)^2}{4N^2}(\max_jp_j^++\min_jp_j^-+2)\geq 2 \beta L\sqrt{N}\dfrac{(N+2)^2}{4N^2}(\sigma^++\sigma^-)>\delta(\sigma^++\sigma^-)(\sigma^-)^{-2}.
\]
According to \eqref{eq:prelim-1}, $\nu$ should be chosen from the inequality

\[
\dfrac{q_j^+}{1-\nu}  < 2+ 2q_i^-\left(\frac{N+2}{2N}-\frac{1}{q_h^-}\right),\qquad i,j=\overline{1,N}.
\]
Such a choice of $\nu$ is possible if

\[
\dfrac{\sigma^+}{1-\nu} < 2+ 2\sigma^-\left(\frac{N+2}{2N}-\frac{1}{q_h^-}\right).
\]
Since

\[
\frac{1}{q_h^-}=\frac{1}{N}\sum_{i=1}^N\frac{1}{q_i^-}<\frac{1}{N}\sum_{i=1}^N\frac{1}{\sigma^-} =\dfrac{1}{\sigma^-},
\]
it is suffient to claim that

\[
\dfrac{\sigma^+}{1-\nu}\leq \frac{N+2}{N}\sigma^-=2+ 2\sigma^-\left(\frac{N+2}{2N}-\frac{1}{\sigma^-}\right) < 2+ 2\sigma^-\left(\frac{N+2}{2N}-\frac{1}{q_h^-}\right).
\]
Solving the first inequality for $\nu$ we obtain $\nu\leq 1-\frac{\sigma^+}{\sigma^-}\frac{N}{N+2}$. Plugging \eqref{eq:mu-new} we obtain
\[
\nu\leq 1- (\mu+\gamma)\frac{N}{N+2}.
\]

The case of a constant vector $\vec p$ follows from the considered one because in this case $\delta=\gamma=0$.
\end{proof}

\begin{prop}
\label{pro:cond-a-b}
Let the conditions of Proposition \ref{pro:choice-1} be fulfilled.
\begin{enumerate}
    \item Let $x_i$ be a fixed direction. If $p_i(x) +r -2 >0$ in $\overline{\Omega}$,
        or
    \begin{equation}\label{strong:osci}
      \dfrac{p^\vee(x)}{p^\wedge(x)}<1+\frac{1}{N},
    \end{equation}
    then \eqref{eq:cond-++} {\rm (a)} holds true.
    \item Condition \eqref{eq:cond-++} {\rm (a)} implies condition \eqref{eq:cond-++} {\rm (b)}.
\end{enumerate}
\end{prop}

\begin{proof}
To prove \eqref{eq:cond-++} {\rm (a)}, it is enough to claim that
\[
q_i(x) < q_h^\ast(x) \quad \text{for} \ x \in \overline{\Omega} \ \text{and} \ \forall \ i \in \overline{1, N}.
\]
Suppose $p_i(x) +r > 2$ for $x \in \overline{\Omega}$. By using \eqref{eq:oscillation}, we obtain the following chain of relations:
\[
\begin{split}
q_i(x) &< q_h^\ast(x) = \frac{N}{\sum_{i=1}^N \frac{1}{q_i(x)} -1} \quad \Leftarrow \quad q^\vee(x) \left(\frac{N}{q^\wedge(x)} -1\right) <N \\
& \Leftarrow \quad N \frac{p^\vee(x)}{p^\wedge(x)} - (p^\vee +r)  <N \quad \Leftarrow \quad N \left( 1+ \frac{2}{N}\right) - (p^\vee +r)  <N \quad \Leftarrow \quad 2< p^\vee(x) +r.
\end{split}
\]
If \eqref{strong:osci} holds, then
\[
\begin{split}
q_i(x) &< q_h^\ast(x) = \frac{N}{\sum_{i=1}^N \frac{1}{q_i(x)} -1} \quad \Leftarrow \quad N \frac{p^\vee(x)}{p^\wedge(x)} - (p^\vee +r)  <N \\
& \quad \Leftarrow \quad \quad N \left( 1+ \frac{1}{N}\right) - (p^\vee +r)  <N \quad \Leftarrow \quad 1< p^\vee(x) +r,
\end{split}
\]
where the last inequality holds trivially. If $q_h^-\geq N$, we may take for $s_j$ an arbitrary positive number. Let $q_h^-<N$. Then, $(q_h^-)^\ast> 2$ because $q_j^-=p_j^-+r>\frac{2N}{N+2}$. By virtue of Proposition \ref{pro:choice-1}

\[
\begin{split}
s_j\leq \frac{\sigma^+}{1-\nu}& \leq \sigma^-\frac{N+2}{N}=2+2\sigma^-\left(\frac{N+2}{2N}-\frac{1}{\sigma^-}\right) <2+2\sigma^-\left(\frac{N+2}{2N}-\frac{1}{q_h^-}\right).
\\
&
\leq 2+2q_i^-\left(\frac{N+2}{2N}-\frac{1}{q_h^-}\right)=2+q_i^-\left(1-\frac{2}{(q_h^-)^\ast}\right).
\end{split}
\]
Since $(q_h^{-})^\ast>2$, condition \eqref{eq:cond-++} (b) follows from the inequality

\[
2+q_i^-\left(1-\frac{2}{(q_h^-)^\ast}\right)<(q_h^-)^\ast\qquad \Leftrightarrow \qquad q_i^-<(q_h^-)^\ast,
\]
which coincides with \eqref{eq:cond-++} (a).
\end{proof}
Summarizing the above arguments we can formulate the following assertion.

\begin{lemma}
\label{le:1-st-integration}
Assume that $\vec p(x)$ satisfies the conditions of Proposition \ref{pro:cond-a-b} (1). Then, $\beta$ can be chosen so small that for every $u\in W^{1,(2,\vec{q}(\cdot))}(K_{b_k})$ with $b_k \leq \beta$, and every $\lambda >0$

\begin{equation}
\label{eq:integration-2}
\int_{K_{b_k}}|u|^{s_j}\,dx\leq \lambda \sum_{i=1}^N \int_{K_{b_k}}|D_i u|^{q_i^-}\,dx+C',\qquad s_j=\frac{q_j^+}{1-\nu},
\end{equation}
with a constant $C$ depending on $\|u\|_{2,K_{b_k}}$, $\lambda$, and $\nu$ defined in Proposition \ref{pro:choice-1}.
\end{lemma}

\begin{cor}
\label{cor:isotrop-1}
If $\vec p$ is a constant vector, then in Proposition \ref{pro:choice-1} $\gamma=0$. If follows that the assertion of Lemma \ref{le:1-st-integration} is true for every $\vec q$, $q_i=p_i+r$ with $r\in (0,1)$, provided that
\[
\mu=\dfrac{p^\vee}{p^\wedge}<1+\frac{1}{N},  \quad \text{or $\quad q_i \geq 2$ and $\mu<1+\frac{2}{N}$}.
\]
\end{cor}

\noindent \textbf{Estimate for $\mathcal{I}_2$.} By Young's inequality, for every $\delta\in (0,1)$

\[
\mathcal{I}_2 \leq \delta\int_\Omega (\epsilon^2+|D_iu|^2)^{\frac{p_i-2}{2}}(D_{ii}^2u)^2\,dx + C_\delta\int_\Omega u^2(\epsilon^2+|D_iu|^2)^{\frac{p_i+2(r-1)}{2}}\,dx,
\]
so that the first term of these estimate can be absorbed in the left-hand side of \eqref{eq:3rd-est}.
Let us take finite cover of $\Omega=K_{\vec a}$ composed of cubes $K_{\vec a,b_k}$. Denote $K_{b_k}\equiv K_{\vec a,b_k}$ and consider the integral

\[
\mathcal{J}=\int_{K_{b_k}} u^2(\epsilon^2+|D_ju|^2)^{\frac{p_j+2(r-1)}{2}}\,dx \equiv \int_{K_{b_k}} u^2(\epsilon^2+|D_ju|^2)^{\frac{q_j+(r-2)}{2}}\,dx.
\]
We assume that the conditions of Lemma \ref{le:1-st-integration} are fulfilled,  and $q_j+r-2>0$, which is true for $r>\dfrac{2}{N+2}$. By Young's inequality, for every $\lambda >0$

\[
\mathcal{J}\leq C(\lambda) \int_{K_{b_k}}|u|^{s_j}\,dx+  \lambda   \int_{K_{b_k}}(\epsilon^2+|D_ju|^2)^{\frac{\kappa_j}{2}}\,dx
\]
with

\[
\frac{\kappa_j}{2}=\frac{q_j+r-2}{2}\dfrac{s_j}{s_j-2}.
\]
The first integral is already estimated in Lemma \ref{le:1-st-integration}. The second integral is bounded if $\kappa_j^+\leq q_j^-$:

\[
q^+_j+r-2\leq \dfrac{q_j^-}{s_j}\left(s_j-2\right)=q_j^--2\dfrac{q_j^-}{s_j},
\]
which is equivalent to

\[
0\leq q_j^+-q_j^-<2-r-2\dfrac{q_j^-}{s_j}= 2-r-2\dfrac{q_j^-}{q_j^+}(1-\nu).
\]
Since $q_j^+-q_j^-\to 0^+$ as $b_k \to 0$, to fulfill this condition suffices to claim that the right-hand side is strictly positive:

\[
r<2\nu\leq 2\left(1-\dfrac{q_j^-}{q_j^+}(1-\nu)\right).
\]
This inequality gives the admissble value of $r$. By Proposition \ref{pro:choice-1}

\[
r<2\nu \leq 2-(\mu +\gamma) \frac{2N}{N+2}.
\]
Comparing the lower bound $r>\dfrac{2}{N+2}$ with the above upper bound, we obtain:

\begin{equation}
\label{eq:r-bounds}
r\in \left(\frac{2}{N+2}, 2-(\mu+\gamma) \frac{2N}{N+2}\right),\quad \text{provided $\;\;\displaystyle \mu+\gamma<1+\frac{1}{N}$}.
\end{equation}

\begin{remark}
\label{rem:improved}
Let us assume that the diffusion in the direction $x_j$ is slow or linear: $p_j(z) \geq 2$ in $\overline{Q}_T$. Then the inequality $p_j+2(r-1)>0$ holds trivially and the admissible value of $r$ is given by
\[
r\in \left(0, 2-(\mu+\gamma) \frac{2N}{N+2}\right),\quad \text{provided $\;\;\displaystyle \mu+\gamma<1+\frac{2}{N}$}.
\]
\end{remark}

\begin{lemma}
\label{le:integration-I-1}
Let the conditions of Lemma \ref{le:1-st-integration} be fulfilled and $\mu<1+\frac{1}{N}$. Then for every cube $K_{b_k}$ with the edge length $b_k \leq \beta$ and $\beta$ so small that
\[
\mu+\gamma<1+\dfrac{1}{N}
\]
with $\gamma$ defined in \eqref{eq:gamma} and every $u\in W^{1,(2,\vec q(\cdot))}(K_{b_k})$ with $q_i(x)=p_i(x)+r$ and $r$ satisfying \eqref{eq:r-bounds} and for every $\lambda >0$

\[
\mathcal{J}\leq  \sum_{j=1}^{N}\int_{K_{b_k}}|D_i u|^{q_i^-}\,dx+C
\]
with constants $C$ depending on $\|u\|_{2,K_{\vec a}}$ and $\lambda.$
\end{lemma}

\begin{cor}
\label{cor:isotrop-2}
If $\vec p$ is a constant vector, Lemma \ref{le:integration-I-1} is true if

\[
\mu=\frac{\max_j p_j}{\min_{j} p_j}<1+\frac{1}{N}\quad \text{and} \quad r\in \left(\frac{2}{N+2},2-\mu\frac{2N}{N+2}\right).
\]
Moreover, if $p_i=p$, then $\mu=1$ and the assertion of Lemma \ref{le:integration-I-1} holds for $q=p+r$ with every

\[
r\in \left(\frac{2}{N+2},\frac{4}{N+2}\right).
\]
\end{cor}

Gathering the estimates of Lemmas \ref{le:1-st-integration}, \ref{le:integration-I-1}, we arrive at the following assertion.

\begin{lemma}
\label{le:high-int-stat}
Let $\|\nabla p_i\|_{\infty,\Omega} \leq L$ for all $i=\overline{1, N}$,

\[
\mu=\sup_\Omega\dfrac{p^\vee(x)}{p^\wedge(x)}<1+\dfrac{1}{N},
\]
Assume $\beta$ is so small that $\mu+\gamma<1+\dfrac{1}{N}$ with $\gamma$ defined in \eqref{eq:gamma}.
Then for every smooth function $u$, every
\[
r\in \left(\frac{2}{N+2}, 2-(\mu+\gamma) \frac{2N}{N+2}\right)
\]
and any $\delta\in (0,1)$ the following inequality holds:

\begin{equation}
\label{eq:high-int-stationary}
\int_{\Omega}(\epsilon^2+|D_iu|^2)^{\frac{p_i+r-2}{2}}|D_iu|^2\,dx\leq \delta \sum_{j=1}^{N} \int_{\Omega}(\epsilon^2+|D_iu|^2)^{\frac{p_i-2}{2}}(D^2_{ij}u)^2\,dx+C
\end{equation}
with a constant $C$ depending on $\delta$, $\beta$, and $\|u\|_{2,\Omega}$.
\end{lemma}

Let us consider the case $p_i\equiv p_i(x,t)$. Accept the notation
\[
S_{b_k,h,\tau}=K_{b_k}\times (\tau,\tau+h).
\]
As before, $K_{b_k}$ are cubes from the finite cover of the domain $K_{\vec a}$.
Divide the interval $(0,T)$ into $m$ sub-intervals $(t_j,t_j+h)$, and represent $Q_T=\Omega\times (0,T)$ as the union of the cylinders $S_{b_k,h,t_l}$. Let

\begin{equation}
\label{eq:S-choice}
\begin{split}
& p_{j}^+=\max_{S_{b_k,h,t_l}}p_j(z),\quad p_{j}^-=\min_{S_{b_k,h,t_l}}p_j(z).
\\
& p^\vee(z)=\max\{p_1(z),\ldots,p_N(z)\},\quad p^\wedge(z)=\min\{p_1(z),\ldots,p_N(z)\},\quad z\in Q_T.
\end{split}
\end{equation}

\begin{thm}
\label{th:higher-integr-par}
Assume that $|p_{it}|+|\nabla p_i|\leq L$ a.e. in $Q_T$ and for all $i=\overline{1, N}$, and

\[
\mu=\sup_{Q_T}\frac{p^\vee(z)}{p^\wedge(z)}<1+\frac{1}{N}.
\]
Then, for every smooth function $u$, every number

\[
r\in \left(0, \frac{2N(1-\mu)+4}{N+2}\right),
\]
and every $\delta\in (0,1)$

\begin{equation}
\label{eq:h-int-par}
\int_{Q_T}(\epsilon^2+|D_iu|^2)^{\frac{p_i(z)+r-2}{2}}|D_iu|^2\,dz\leq \delta \sum_{j=1}^{N} \int_{Q_T}(\epsilon^2+|D_iu|^2)^{\frac{p_i(z)-2}{2}}(D^2_{ij}u)^2\,dz+C
\end{equation}
with a constant $C$ depending on $\delta$, $L$, and $\operatorname{ess}\sup_{(0,T)}\|u\|_{2,\Omega}$.
\end{thm}

\begin{proof}
It is sufficient to prove \eqref{eq:h-int-par} for $r>\frac{2}{N+2}$, the case $r\in \left(\left.0,\frac{2}{N+2}\right.\right]$ follows then by Young's inequality.  The proof imitates the proof of Lemma \ref{le:high-int-stat}. The difference consists in the choice of parameters $b_k$ and $h$, which now should be chosen so small that condition \eqref{eq:oscillation-q} is fulfilled in every cylinder $S_{b_k,h,t_l}$ of the partition of $Q_T$. For every cylinder $S_{b_k,h,t_l}$ we take

\[
\gamma=L\sqrt{4N\beta^2+h^2}\dfrac{(N+2)^2}{4N^2}(\max_jp_j^++\min_jp_j^-+2)
\]
with $p_j^\pm$ defined in \eqref{eq:S-choice}, and choose $\beta$, $h$ as small as is needed to obtain $\mu+\gamma<1+\dfrac{1}{N}$ in every cylinder $S_{b_k,h,t_l}$. This leads to estimate \eqref{eq:high-int-stationary} in every cube $K_{b_k}$ for a fixed $t\in (t_l,t_l+h)$. The conclusion follows then upon integration of \eqref{eq:high-int-stationary} in $t$ over the intervals $(t_l,t_l+h)$ and summation of the results.
\end{proof}

\begin{remark}
\label{rem:improved-cylinder}
The assertions of Lemmas \ref{le:integration-I-1}, \ref{le:high-int-stat} and Theorem \ref{th:higher-integr-par} remain true if $p_i(z)\geq 2$ in $\overline{Q}_T$ and $\mu<1+\dfrac{2}{N}$ - see Remark \ref{rem:improved}.
\end{remark}

\section{Existence, uniqueness and regularity of solutions}
\label{sec:proofs}
\subsection{Proof of Theorem \ref{th:existence-reg}}
Theorem \ref{th:higher-integr-par} allows for the following refinement of the a priori estimates of Lemmas \ref{le:a-priori-1}, \ref{le:a-priori-2}, \ref{le:a-priori-3}.

\begin{lemma}
\label{le:refine}
Let $\Omega=K_{\vec a}$. Assume that $\vec p(z)$ satisfies the conditions of Theorem \ref{th:higher-integr-par}. Then, the finite-dimensional approximations $u\equiv u_\epsilon^{(m)}$ satisfy the following uniform estimate

\begin{equation}
\label{eq:refinement}
\begin{split}
\|u_t\|_{2,Q_T}^2 & +\sup_{(0,T)}\| u\|^2_{2,\Omega} +\sup_{(0,T)}\|\nabla u\|^2_{2,\Omega} +\sum_{i=1}^N\sup_{(0,T)}\int_{\Omega}(\epsilon^2+|D_iu|^2)^{\frac{p_i(z)}{2}} \,dx
\\
& +\sum_{i,j=1}^N\int_{Q_T}(\epsilon^2+|D_iu|^2)^{\frac{p_i(z)-2}{2}} \left(D^2_{ij}u\right)^2\,dz +\sum_{i=1}^{N}\int_{Q_T}|D_i u|^{p_i(z)+r}\,dz
\\
&
\leq C\left(1+\|f\|_{2,Q_T}^2 + \|\nabla f\|_{2,Q_T}^2
+ \sum_{i=1}^N\int_{\Omega}|D_iu_0|^{p_i(x,0)}\,dx + \|u_0\|_{W_0^{1,2}(\Omega)}^2\right)
\end{split}
\end{equation}
with any
\[
r\in \left(\frac{2}{N+2},\frac{2N(1-\mu)+4}{N+2}\right).
\]
The constant $C$ is independent of $m$ and $\epsilon$.
\end{lemma}
Estimate \eqref{eq:refinement} allows one to find functions $u_\epsilon$, $\eta_j$, and a subsequence of $\{u_\epsilon^{(m)}\}$ such that

\begin{equation}
\label{eq:conv-eps}
\begin{split}
& \text{$u_{\epsilon}^{(m)}\to u_{\epsilon}$ $\ast$-weak in $L^\infty(0,T;L^2(\Omega))$},
\\
&
\text{$\nabla u_\epsilon^{(m)}\to \nabla u_\epsilon$ $\ast$-weak in $L^\infty(0,T;L^2(\Omega))^N$},
\\
& \text{$u_{\epsilon t}^{(m)}\rightharpoonup u_{\epsilon t}$ in $L^2(Q_T)$},
\\
& \text{$D_ju_\epsilon^{(m)}\rightharpoonup D_ju_\epsilon$ in $L^{p_j(\cdot)}(Q_T)$, }
\\
& \text{$\mathcal{F}_{\epsilon}^{(j)}(z,D_ju_\epsilon^{(m)})\rightharpoonup  \eta_j$ in $L^{p'_j(\cdot)}(Q_T)$}.
\end{split}
\end{equation}
Since $W^{1,\vec p(\cdot)}_0(\Omega)\subset W^{1,p^\wedge}_0(\Omega)\hookrightarrow L^2(\Omega)$, the functions $u_\epsilon^{(m)}$ are uniformly bounded in $L^\infty(0,T;W^{1,p^-}_0(\Omega))$. Since $u_{\epsilon t}^{(m)}$ are uniformly bounded in $L^2(Q_T)$, the sequence $\{u_\epsilon^{(m)}\}$ is relatively compact in $C^0([0,T];L^2(\Omega))$, see \cite[Sec.8, Cor.4]{Simon-1987}. Thus, $u_\epsilon^{(m)}\to u_\epsilon$ a.e. in $Q_T$.

By the method of construction, for every $k\leq m$

\begin{equation}
\label{eq:def-approx}
\int_{Q_T}\left(u_{\epsilon t}^{(m)}\phi+\sum_{j=1}^N \mathcal{F}_{\epsilon}^{(j)}(z,D_ju_{\epsilon}^{(m)})D_j\phi\right)\,dx = \int_{Q_T}f \phi\,dz,\qquad \forall \phi\in \mathcal{S}_k.
\end{equation}
Letting $m\to \infty$ we obtain the equality

\begin{equation}
\label{eq:ident-provisional}
\int_{Q_T}\left(u_{\epsilon t}\phi+ \sum_{j=1}^N\eta_j D_j\phi\right)\,dz= \int_{Q_T}f \phi\,dz
\end{equation}
with any $\phi\in \mathcal{S}_k$ with a fixed $k$. Since $\mathbb{W}(Q_T)=\bigcup_{k\geq 1}\mathcal{S}_k$, the same is true for every $\phi\in \mathbb{W}(Q_T)$. The functions $\mathcal{F}^{(j)}_\epsilon(z,\xi)$ are monotone: for all $\xi,\eta\in \mathbb{R}$

\begin{equation}
\label{eq:mon}
(\mathcal{F}_{\epsilon}^{(j)}(z,\xi)-\mathcal{F}_{\epsilon}^{(j)}(z,\eta))(\xi-\eta)\geq C\begin{cases}
|\xi-\eta|^p & \text{if $p\geq 2$},
\\
|\xi-\eta|^2(\epsilon^2+|\xi|^2+|\eta|^2)^{\frac{p-2}{2}} & \text{if $p\in (1,2)$}
\end{cases}
\end{equation}
with an absolute constant $C$. By using monotonicity and density of $\bigcup_{k\geq 1}\mathcal{S}_k$ in $\mathbb{W}(Q_T)$ we identify $\eta_j$ by the standard arguments (see, e.g., \cite[Sec.6]{A-S-2020}):
\[
\int_{Q_T}\eta_jD_j\phi\,dz=\int_{Q_T} \mathcal{F}_{\epsilon}^{(j)}(z,D_iu_\epsilon)D_j\phi\,dz\qquad \forall \phi\in \mathbb{W}(Q_T).
\]
It follows that the limit $u_\epsilon$ is a solution of problem \eqref{eq:main-reg}. Moreover, the uniform estimate \eqref{eq:refinement} entails the estimate

\begin{equation}
\label{eq:high-par-reg}
\sum_{i=1}^{N}\int_{Q_T}|D_iu_\epsilon|^{p_i(z)+r}\,dz\leq C
\end{equation}
with an independent of $\epsilon$ constant $C$.

Uniqueness of the constructed solution is an immediate byproduct of the monotonicity \eqref{eq:mon}: testing \eqref{eq:def-main-reg} for the solutions $u_{\epsilon, 1}$, $u_{\epsilon, 2}$ with $\phi=u_{\epsilon, 1}-u_{\epsilon, 2}$ we obtain the inequality $\|u_{\epsilon, 1}-u_{\epsilon, 2}\|_{2,\Omega}^2(t)\leq 0$ for a.e. $t\in (0,T)$.

\subsection{Strong convergence of the gradients}\label{a.e.-conv}

The strong convergence $u_\epsilon^{(m)}\to u_\epsilon$ in $L^2(Q_T)$, the weak convergence $D_ju_\epsilon^{(m)}\rightharpoonup D_ju_\epsilon$ in $L^{p_j(\cdot)}(Q_T)$ and the Mazur Lemma (see \cite[Ch.3,Cor.3.8]{Bresis-book}) yield the existence of a sequence $\{v^{(m)}\}$ of convex combinations of $\{u_\epsilon^{(k)}\}_{k=1}^{m}$ such that $v^{(m)}\to u_\epsilon$ in $\mathbb{W}(Q_T)$. Let us define $w_m\in \mathcal{S}_m$ as follows:

\[
\|w_m-u_\epsilon\|_{\mathbb{W}(Q_T)}=\min\left\{\|w-u_\epsilon\|_{\mathbb{W}(Q_T)}: w\in \mathcal{S}_m\right\}.
\]
The functions $u_{\epsilon}^{(m)}, w_m\in \mathcal{S}_m$ are admissible test-functions in \eqref{eq:def-main-reg} and \eqref{eq:def-approx}. Combining these equalities with the test-function $u_\epsilon^{(m)}$ we obtain

\[
\begin{split}
\sum_{i=1}^N\mathcal{K}_i & \equiv \sum_{i=1}^N \int_{Q_T}(\mathcal{F}_\epsilon^{(i)}(z,D_iu^{(m)}_\epsilon) -\mathcal{F}_\epsilon^{(i)}(z,D_iu_\epsilon))(D_i u_{\epsilon}^{(m)}-D_i u_{\epsilon})\,dz
\\
& =-\int_{Q_T}(u_{\epsilon t}^{(m)}-u_{\epsilon t})u_{\epsilon}^{(m)}\,dz
- \sum_{i=1}^N \int_{Q_T}(\mathcal{F}_\epsilon^{(i)}(z,D_iu^{(m)}_\epsilon) -\mathcal{F}_\epsilon^{(i)}(z,D_iu_\epsilon))D_i u_{\epsilon}\,dz.
\end{split}
\]
Choosing $w_m$ for the test-function we have the equality
\[
-\int_{Q_T}(u_{\epsilon t}^ {(m)}-u_{\epsilon t})w_m\,dz - \sum_{i=1}^{N}\int_{Q_T}(\mathcal{F}_\epsilon^{(i)}(z,D_iu^{(m)}_\epsilon) -\mathcal{F}_\epsilon^{(i)}(z,D_iu_\epsilon))D_iw_m\,dz=0.
\]
Subtracting the second equality from the first one we find that

\[
\begin{split}
\sum_{i=1}^N\mathcal{K}_i & \leq  \|u_{\epsilon t}^{(m)}-u_{\epsilon t}\|_{2,Q_T}\left(\|u_{\epsilon}-w_m\|_{2,Q_T} +\|u_{\epsilon}^{(m)}-u_{\epsilon}\|_{2,Q_T}\right)
\\
& +\sum_{i=1}^{N}\left(\|\mathcal{F}_\epsilon^{(i)}(z,D_iu^{(m)}_\epsilon)\|_{p_i'(\cdot),Q_T}+ \|\mathcal{F}_\epsilon^{(i)}(z,D_iu_\epsilon)\|_{p_i'(\cdot),Q_T}\right)\|D_iu_\epsilon-D_iw_m\|_{p_i(\cdot),Q_T}.
\end{split}
\]
By the choice of $\{w_m\}$ both terms on the right-hand side tend to zero as $m\to \infty$. Fix $i=\overline{1,N}$ and consider $\mathcal{K}_i$. Let us denote $Q_i^+=Q_T\cap \{p_i(z)\geq 2\}$ and $Q_i^-=Q_T\setminus Q_i^+$. By \eqref{eq:mon}

\[
\int_{Q_i^+}(\mathcal{F}_\epsilon^{(i)}(z,D_iu^{(m)}_\epsilon) -\mathcal{F}_\epsilon^{(i)}(z,D_iu_\epsilon))(D_i u_{\epsilon}^{(m)}-D_i u_{\epsilon})\,dz\geq \int_{Q_i^+}|D_i u_\epsilon^{(m)}-D_i u_\epsilon|^{p_i(z)}\,dz.
\]
On $Q_i^-$ we apply the generalized H\"older inequality \eqref{eq:Holder}:

\[
\begin{split}
\int_{Q_i^-} & |D_i(u_\epsilon^{(m)}-u_\epsilon)|^{p_i(z)}\,dz \equiv  \int_{Q_i^-} \left(\dfrac{(\epsilon^2+|D_iu_{\epsilon}^{(m)}|^2+|D_iu_{\epsilon}|^2)^{\frac{p_i-2}{2}} |D_i(u_\epsilon^{(m)}- u_\epsilon)|^2}{(\epsilon^2+|D_iu_{\epsilon}^{(m)}|^2 +|D_iu_{\epsilon}|^2)^{\frac{p_i-2}{2}}}\right)^{\frac{p_i}{2}}\,dz
\\
& \leq 2\left\|\left((\epsilon^2+|D_iu_{\epsilon}^{(m)}|^2+|D_iu_{\epsilon}|^2)^{\frac{p_i-2}{2}} |D_i(u_\epsilon^{(m)}- u_\epsilon)|^2\right)^{\frac{p_i}{2}}\right\|_{\frac{2}{p_i(\cdot)},Q_i^-}
\\
& \qquad \times \left\|(\epsilon^2+|D_iu_{\epsilon}^{(m)}|^2 +|D_iu_{\epsilon}|^2)^{p_i\frac{2-p_i}{4}}\right\|_{\frac{2}{2-p_i(\cdot)},Q_i^-}.
\end{split}
\]
Due to relation \eqref{eq:norm-mod} between the norm and the modular in the variable Lebesgue space, the second factor is bounded by a constant independent of $m$ and $\epsilon$. Because of \eqref{eq:mon} with $p\in (1,2)$ and \eqref{eq:norm-mod}, the first factor tends to zero as $\mathcal{K}_i\to 0$. It follows that  $\|D_i(u_\epsilon^{(m)}-u_\epsilon)\|_{p_i(\cdot),Q_T}\to 0$ as $m\to \infty$, which yields the pointwise convergence $D_iu_\epsilon^{(m)}\to D_iu_\epsilon$ a.e. in $Q_T$. By \eqref{eq:refinement} and the Vitali convergence theorem $D_iu_\epsilon^{(m)}\to D_iu_\epsilon$ in $L^{p_i(\cdot)+r}(Q_T)$ with every $r$ from the conditions of Lemma \ref{le:refine}.

\subsection{Proof of Theorem \ref{th:high-diff-reg}}
For all $i,j=\overline{1,N}$

\[
\begin{split}
D_j \left((\epsilon^2+|D_i u_\epsilon^{(m)}|^2)^{\frac{p_i-2}{4}}D_{i}u_\epsilon^{(m)} \right) & = (\epsilon^2+|D_i u_\epsilon^{(m)}|^2)^{\frac{p_i-2}{4}}D^{2}_{ij}u_\epsilon^{(m)}
\\
& + \frac{p_i-2}{2}(\epsilon^2+|D_i u_\epsilon^{(m)}|^2)^{\frac{p_i-2}{4}-1}(D_iu_\epsilon^{(m)})^2D^2_{ij}u_\epsilon^{(m)}
\\
& + (\epsilon^2+|D_i u_\epsilon^{(m)}|^2)^{\frac{p_i-2}{4}}D_{i}u_\epsilon^{(m)} \frac{1}{4}D_jp_i \ln (\epsilon^2+|D_i u_\epsilon^{(m)}|^2),
\end{split}
\]
whence, with the use of \eqref{eq:elem-1},

\[
\begin{split}
\left|D_j \left(\mathcal{F}_i^{(\epsilon)}(z,D_iu_\epsilon^{(m)})\right)\right|
\leq C (\epsilon^2+|D_i u_\epsilon^{(m)}|^2)^{\frac{p_i-2}{4}}|D_{ij}^2 u_\epsilon^{(m)}|
+ C'(\epsilon^2+|D_i u_\epsilon^{(m)}|^2)^{\frac{p_i}{4}+\frac{\rho}{2}} + C''.
\end{split}
\]
By virtue of \eqref{eq:refinement}

\begin{equation}
\label{eq:2nd-order-m-eps}
\left\|D_j \left(\mathcal{F}_i^{(\epsilon)}(z,D_iu_\epsilon^{(m)})\right)\right\|_{2,Q_T}\leq C.
\end{equation}
It follows that there exist $\zeta_{ij}\in L^2(Q_T)$ such that (up to a subsequence)

\[
D_j \left(\mathcal{F}_i^{(\epsilon)}(z,D_iu_\epsilon^{(m)} )\right) \rightharpoonup \zeta_{ij}\quad \text{in $L^2(Q_T)$ as $m\to \infty$}.
\]
Due to a.e. convergence $D_iu_\epsilon^{(m)}\to D_iu_\epsilon$, for every $\phi\in C_0^\infty(Q_T)$

\[
\begin{split}
(\zeta_{ij},\phi)_{2,Q_T} & =\lim_{m\to \infty}\left(D_j\left(\mathcal{F}_i^{(\epsilon)}(z,D_iu^{(m)}_\epsilon)\right),\phi\right)_{2,Q_T}
\\
& = -\lim_{m\to \infty}(\mathcal{F}_i^{(\epsilon)}(z,D_iu^{(m)}_\epsilon),D_j\phi)_{2,Q_T}
 =- (\mathcal{F}_i^{(\epsilon)}(z,D_iu_\epsilon),D_j\phi)_{2,Q_T}.
\end{split}
\]
Thus, $\zeta_{ij}=D_j\mathcal{F}_i^{(\epsilon)}(z,D_iu_\epsilon)$, and estimate \eqref{eq:2nd-order-m-eps} holds for $D_j\mathcal{F}_i^{(\epsilon)}(z,D_iu_\epsilon)$.

\begin{remark}\label{eq:second-der:fast-0}
Assume that for some $p_j(z) <2$ in $Q_T$ for some $j$. By Young's inequality,
\[
\begin{split}
\int_{Q_T}|D^2_{ij}u_\epsilon^{(m)}|^{p_j}\,dz & = \int_{Q_T}(\epsilon^2+(D_ju_\epsilon^{(m)})^2)^{\frac{p_j(2-p_j)}{4}} \left((\epsilon^2+(D_ju_\epsilon^{(m)})^2)^{\frac{p_j-2}{2}}|D_{ij}^2u_\epsilon^{(m)}|^2\right)^{\frac{p_j}{2}}\,dz\\
& \leq \int_{Q_T}  (\epsilon^2+(D_ju_\epsilon^{(m)})^2)^{\frac{p_j-2}{2}}|D_{ij}^2u_\epsilon^{(m)}|^2 \,dz + C \int_{Q_T}(\epsilon^2+(D_ju_\epsilon^{(m)})^2)^{\frac{p_j}{2}}\,dz.\\
\end{split}
\]
By virtue of \eqref{eq:refinement}, the integrals on the right-hand side of this inequality are uniformly bounded  with respect to $m$ and $\epsilon$. It follows that there exist $\zeta_{ij} \in L^{p_j(\cdot)}(Q_T)$ such that $D^2_{ij}u_\epsilon^{(m)} \rightharpoonup \zeta_{ij}$ in $ L^{p_j(\cdot)}(Q_T)$  (up to a subsequence). Since $D_j u_\epsilon^{(m)} \rightharpoonup D_j u_\epsilon$ in $L^{2}(Q_T)$, then for every $\phi
\in C_0^\infty(Q_T)$
\[
\begin{split}
(\zeta_{ij},\phi)_{2,Q_T} & =\lim_{m\to \infty}\left(D^2_{ij} u^{(m)}_\epsilon) ,\phi\right)_{2,Q_T}
\\
& = -\lim_{m\to \infty}(D_ju^{(m)}_\epsilon),D_i\phi)_{2,Q_T} = ( D_j u_\epsilon,D_i\phi)_{2,Q_T}.
\end{split}
\]
This further implies that $\zeta_{ij} = D^2_{ij} u_\epsilon$ and by lower semicontinuity of the modular,  $\|D^2_{ij} u_\epsilon\|_{p_j(\cdot), Q_T} \leq C.$
\end{remark}

\subsection{Proof of Theorems \ref{th:existence}, \ref{th:high-order-reg}}
The proofs imitate the proofs of the corresponding assertions for the regularized problem \eqref{eq:main-reg}. Let $\{u_\epsilon\}$ be the family of solutions of problem \eqref{eq:main-reg}. The uniform estimates \eqref{eq:est-epsilon} and \eqref{eq:high-reg} allow us to choose a sequence $\{u_{\epsilon_k}\}$ (we simply write $u_\epsilon$) which has the convergence properties \eqref{eq:conv-eps} with some $u\in \mathbb{W}(Q_T)$ and $\eta_j\in L^{p_j}(Q_T)$.
To pass to the limit as $\epsilon\to 0$ one may repeat the arguments of \cite[Sec.7]{A-S-2020}. The higher integrability of $D_iu$ follows from the independent of $\epsilon$ estimate \eqref{eq:high-par-reg}.

The strong convergence of the gradients follows from the strong monotonicity of the flux functions. Let $u_\epsilon$, $u_\delta$ be two solutions of the regularized problem \eqref{eq:main-reg}. Combining equalities \eqref{eq:def-main-reg} with the test-function $\phi=u_\epsilon-u_\delta\in \mathbb{W}(Q_T)$ we write the result in the form

\[
\begin{split}
\sum_{i=1}^N & \int_{Q_T}\left(\mathcal{F}_i^{(\epsilon)}(z,D_iu_\epsilon)-\mathcal{F}_i^{(\epsilon)}(z,D_iu_\delta)\right) \left(D_iu_\epsilon - D_i u_\delta\right)\,dz
= \int_{Q_T}(u_{\epsilon}-u_\delta)_t (u_{\epsilon}-u_\delta)\,dz
\\
&
- \sum_{i=1}^N\int_{Q_T}\left((\delta^2+|D_iu_\delta|^2)^{\frac{p_i-2}{2}} -(\epsilon^2+|D_iu_\delta|^2)^{\frac{p_i-2}{2}}\right)D_iu_\delta\left(D_i u_\epsilon- D_iu_\delta\right) \,dz.
\end{split}
\]
Due to \eqref{eq:conv-eps}, the first term on the right-hand side tends to zero as $\epsilon,\delta\to 0$ as the product of the weakly convergent sequence, $(u_\epsilon-u_\delta)_t\rightharpoonup 0$ in $L^2(Q_T)$, and the strongly convergent sequence $u_\epsilon-u_\delta\to 0$ in $L^2(Q_T)$. The second term tends to zero by the Vitali convergence theorem. On the one hand, the integrand tends to zero a.e. in $Q_T$ as $\epsilon-\delta\to 0$. On the other hand,  due to \eqref{eq:high-par-reg} it belongs to $L^{1+\sigma}(Q_T)$ with a sufficiently small $\sigma>0$. Indeed: by Young's inequality

\[
\begin{split}
& \left|\mathcal{F}_i^{(\delta)}(z,D_iu_\delta)-\mathcal{F}_i^{(\epsilon)}(z,D_iu_\delta)\right|^{1+\sigma} |D_i(u_\epsilon-u_\delta)|^{1+\sigma}
\\
& \qquad \leq C\left(1+|D_iu_\delta|^{(p_i-1)(1+\sigma)}\right)\left(|D_iu_\epsilon|^{1+\sigma}+|D_iu_\delta|^{1+\sigma}\right)
\\
& \qquad \leq C'\left(|D_iu_\epsilon|^{p_i(1+\sigma)}+|D_iu_\delta|^{p_i(1+\sigma)}\right) +C''\left(1+|D_iu_\delta|^{p_i(1+\sigma)}\right),
\end{split}
\]
therefore the integrand of the second term belongs to $L^{1+\sigma}(Q_T)$, provided $\sigma$ is so small that $\sigma \max_{\overline{Q}_T}p^\vee(z)<\frac{2}{N+2}$. Using \eqref{eq:mon} to estimate the left-hand side from below and then arguing as in the proof of the strong convergence of the gradients for the solutions of the regularized problems, we conclude that $D_{i}(u_\epsilon-u_\delta)\to 0$ in $L^{p_i(\cdot)}(Q_T)$ and, thus, a.e. in $Q_T$. The inclusions $|D_iu|^{\frac{p_i(z)-2}{2}}D_iu\in W^{1,2}(Q_T)$ follow now as in the proof of Theorem \ref{th:high-diff-reg}.

\begin{remark}\label{eq:second-der:fast-1}
The assertion of Remark \ref{eq:second-der:fast} follows by repeating the same arguments for $u_\epsilon$ in place of $u_\epsilon^{(m)}$ as in Remark \ref{eq:second-der:fast-0} and in the proof of Theorem \ref{th:high-diff-reg}.
\end{remark}

\section{Problems \eqref{eq:main}, \eqref{eq:main-reg} in a smooth domain. Proof of Theorem \ref{th:smooth-domain}}
\label{sec:smooth-domain}
We turn to the problems posed in a cylinder $Q_T=\Omega\times (0,T)$ with $\partial\Omega\in C^k$, $k\geq 1+N\left(\frac{1}{2}-\frac{1}{p^+}\right)$. The proof of Theorem \ref{th:smooth-domain} is an imitation of the proofs of the corresponding assertions in the rectangular domains. For this reason, we omit the details and present the arguments where the role of the domain geometry becomes crucial.

The solution of problem \eqref{eq:main-reg} is sought as the limit of the sequence $u^{(m)}_\epsilon\in \mathcal{S}_m$ in the basis composed of solutions of problem \eqref{eq:eigen-smooth}. For every $m\in \mathbb{N}$ the coefficients $c_i^{(m)}(t)$ are defined as the solutions of problems \eqref{eq:ODE-m}, the existence of a sequence $u^{(m)}(x,0)\to u_0$ in $W^{1,\vec q(x)}_0(\Omega)$ with $q_i(x)=\max\{2,p_i(x,0)\}$ follows from Lemma \ref{le:density-smooth-boundary}. The energy relation \eqref{eq:1-est} and the a priori estimates \eqref{eq:1st-est}, \eqref{eq:2nd-est} for the approximations $u\equiv u^{(m)}_\epsilon$ do not change if the rectangular domain $K_{\vec a}$ is substituted by a smooth domain $\partial\Omega$,

\begin{lemma}
\label{le:1-2-smooth}
Let $\partial\Omega\in C^{k}$ with $k\geq 1+N\left(\frac{1}{2}-\frac{1}{p^+}\right)$. If $p_i(z)$ satisfy the conditions of Theorem \ref{th:existence}, then the approximations $u\equiv u^{(m)}_\epsilon$ satisfy the uniform estimates \eqref{eq:1st-est}, \eqref{eq:2nd-est}.
\end{lemma}

The difference between the cases of rectangular and smooth domains reveals in the derivation of the analog of equality \eqref{eq:3rd-est}. The boundary integrals, that were vanishing due to the geometry of the domain $K_{\vec a}$, are now present and lead to new restrictions on the admissible anisotropy of the equation. We will follow the proof of \cite[Th.3.1.1.1]{Grisvard-book}, which allows one to present the estimate in the form independent of the particular choice of the parametrization of $\partial\Omega$. Let us take a point $\xi\in \partial\Omega$ and choose $C^2$ curves $\{l_1,\ldots,l_{N-1}\}$ that are orthogonal at $\xi$. Denote by $\{\vec \tau_1,\ldots,\vec \tau_{N-1}\}$ the unit vectors tangent to $l_i$, and by $s_i$ the curve length along the curve $l_i$. By $\vec \nu$ we denote the normal vector to $\partial\Omega$. Given a smooth vector $\vec v$, we decompose it on $\partial\Omega$ into the sum of the tangent and normal components:

\[
\begin{split}
& \vec v=\vec v_\tau +v_\nu\vec \nu,
\quad  \vec v_\tau=\sum_{j=1}^{N-1} v_j\vec \tau_j,\quad v_j=(\vec v,\vec \tau_j).
\end{split}
\]
Let $\vec v$ and $\vec w$ be two given smooth vectors. Integrating by parts and literally following the proof of \cite[Th.3.1.1.1]{Grisvard-book} we arrive at the formula

\begin{equation}
\label{eq:G-1}
\int_\Omega \operatorname{div}\vec v \operatorname{div}\vec w\,dx = \int_{\partial\Omega}\left(v_\nu \operatorname{div}\vec w- ((\vec v\cdot \nabla)\vec w)\cdot \vec \nu\right)\,dS + \int_\Omega \sum_{i,j=1}^N D_j v_i D_i w_j\,dx,
\end{equation}
(cf. with \eqref{eq:3rd-est-pre}), where the boundary integral can be reduced to the form

\begin{equation}
\label{eq:G-2}
\begin{split}
\int_{\partial\Omega} \left(v_\nu \operatorname{div}\vec w- ((\vec v\cdot \nabla)\vec w)\cdot \vec \nu\right)\,dS & = -\int_{\partial\Omega} \left(\vec v_\tau \nabla _\tau (\vec w\cdot \vec\nu)+ \vec w_\tau \nabla _\tau (\vec v\cdot \vec\nu)\right)\,dS
\\
&
- \int_{\partial\Omega}\mathcal{B}(\vec v_\tau;\vec w_\tau)\,dS-\int_{\partial\Omega}v_{\nu}w_{\nu}\operatorname{tr}\mathcal{B} \,dS.
\end{split}
\end{equation}
Here $\mathcal{B}$ is the matrix of the second quadratic form of the surface $\partial\Omega$. In the local coordinates $\{y_i\}$ with the origin $\xi$ the surface $\partial\Omega$ is represented by the equation $y_N=\phi(y_1,\ldots,y_{N-1})$ where $(y_1,\ldots,y_{N-1})$ belongs to the tangent plane and $y_N$ points in the direction of the exterior normal $\vec\nu$. For every two vectors $\zeta$, $\eta$ tangent to $\partial\Omega$ at the point $\xi\in \partial\Omega$
\[
\mathcal{B}(\xi;\eta)=\sum_{i,j=1}^{N-1}D^2_{y_iy_j}\phi(\xi)\zeta_i\eta_j,\qquad \operatorname{tr}\mathcal{B}=\sum_{i=1}^{N-1}D_{y_iy_i}^2\phi(\xi).
\]

\begin{lemma}
\label{le:by-parts-smooth}
Let the conditions of Lemma \ref{le:1-2-smooth} be fulfilled and, in addition, $p_i(z)=2$ of $\partial\Omega\times [0,T]$. Then the functions $u\equiv u_\epsilon^{(m)}$ satisfy the uniform estimates

\begin{equation}
\label{eq:4-est}
\begin{split}
\sup_{(0,T)}\|\nabla u(t)\|^2_{2,\Omega} & + \sum_{i,j=1}^N\int_{Q_T}(\epsilon^2+|D_iu|^2)^{\frac{p_i(z)-2}{2}} \left(D^2_{ij}u\right)^2\,dz
\\
&
\leq C\left(1+\sum_{i=1}^N\int_{Q_T}|D_iu|^{p_i(z)+\rho}\,dz + \|\nabla u_0\|_{2,\Omega}^2+\| \nabla f\|_{2,Q_T}^{2}\right)
\end{split}
\end{equation}
with any $\rho\in (0,1)$ and a constant $C$ independent of $\epsilon$, $m$.
\end{lemma}

\begin{proof}
We use the notation $u\equiv u_{\epsilon}^{(m)}$. Multiplying the $k$th equation in \eqref{eq:ODE-m} by $\lambda_k c_{m,k}(t)$, integrating by parts, and using \eqref{eq:G-1}, \eqref{eq:G-2} with $ v_i= \mathcal{F}^{(\epsilon)}_{i}(z, D_i u)$ and $\vec w=\nabla u$, we arrive at the equality

\begin{equation}
\label{eq:inter-smooth}
\begin{split}
\dfrac{1}{2}\dfrac{d}{dt}\|\nabla u\|_{2,\Omega}^2 & +\sum_{i,j=1}^N\int_{\Omega}D_j\left((\epsilon^2+(D_iu)^2)^{\frac{p_i(z)-2}{2}}\right)D^2_{ij}u\,dx - \int_{\Omega}\nabla f\cdot \nabla u\,dx
\\
& = \int_{\partial\Omega}\vec v_\tau \nabla _\tau (\vec w\cdot \vec\nu)\,dS+ \int_{\partial\Omega}\vec w_\tau \nabla _\tau (\vec v\cdot \vec\nu)\,dS
+ \int_{\partial\Omega}\mathcal{B}(\vec v_\tau;\vec w_\tau)\,dS+\int_{\partial\Omega}v_{\nu}w_{\nu}\operatorname{tr}\mathcal{B} \,dS.
\end{split}
\end{equation}
The second and the third terms on the right-hand side vanish because $\vec w_\tau=\nabla_\tau u=0$ on $\partial\Omega$. To eliminate the first term we claim that $\vec v_\tau=0$, that is, at every point of $\partial\Omega$ the vector with the components $\mathcal{F}^{(\epsilon)}_{i}(z, D_i u)$ (the flux) either equals zero, or points in the direction of the normal $\vec \nu$ to $\partial\Omega$. Since $\vec \nu=\dfrac{\nabla u}{|\nabla u|}$ at the points where $|\nabla u|\not=0$, this is true if $p_i(z)=2$ on $\partial\Omega$ for all $i=\overline{1,N}$: $\mathcal{F}^{(\epsilon)}_{i}(z, D_i u)=D_iu$. The last term is then bounded by

\[
 K\int_{\partial\Omega}|\nabla u|^2\,dS\leq K\sum_{i=1}^N\int_{\partial\Omega}(\epsilon^2+|D_iu|^2)^{\frac{p_i(z)}{2}}\,dS
\]
with a constant $K$ depending on the main curvature of $\partial\Omega$. By \cite[Lemma 1.5.1.9]{Grisvard-book} there exists a function $\vec \mu\in C^{\infty}(\overline{\Omega})^N$ such that $\vec \mu\cdot\vec \nu \geq \delta>0$ on $\partial\Omega$ for some constant $\delta$ depending on $\partial\Omega$. Then

\[
\begin{split}
\delta\int_{\partial\Omega}(\epsilon^2+|D_iu|^2)^{\frac{p_i(z)}{2}}\,dS &  \leq \int_{\Omega}\operatorname{div}\left((\epsilon^2+|D_iu|^2)^{\frac{p_i(z)}{2}}\vec \mu\right)\,dx
 \\
 &
 = \int_{\Omega} \vec \mu \cdot\nabla \left((\epsilon^2+|D_iu|^2)^{\frac{p_i(z)}{2}}\right)\,dx + \int_{\Omega} (\epsilon^2+|D_iu|^2)^{\frac{p_i(z)}{2}} \operatorname{div}\vec \mu\,dx.
 \end{split}
 \]
 The second term on the right-hand side is bounded by

 \[
 C\max_{\overline\Omega}|\operatorname{div}\vec \mu| \left(1+\int_{\Omega}|D_iu|^{p_i(z)}\,dx\right)\leq C'+C'' \int_{\Omega}|D_iu|^{p_i(z)}\,dx,
 \]
By applying \eqref{eq:elem-1} and Young's inequality, we estimate the first term as follows:

\[
\begin{split}
C\sum_{j=1}^{N} & \int_{\Omega} p_i(z)(\epsilon^2+|D_iu|^2)^{\frac{p_i(z)-1}{2}}|D_{ij}^{2}u|\,dx + C\int_{\Omega}(\epsilon^2+|D_iu|^2)^{\frac{p_i(z)}{2}}|\ln (\epsilon^2+|D_iu|^2)||\nabla p_i|\,dx
\\
& \leq C'\sum_{j=1}^N \int_\Omega (\epsilon^2+|D_iu|^2)^{\frac{p_i(z)}{4}}\left((\epsilon^2+|D_iu|)^{\frac{p_i(z)-2}{2}} (D^2_{ij}u)^2\right)^\frac{1}{2}\,dx + C'\int_{\Omega}(\epsilon^2+|D_iu|^2)^{\frac{p_i(z)+\rho}{2}} \,dx
\\
& \leq \lambda \sum_{j=1}^N \int_\Omega (\epsilon^2+|D_iu|)^{\frac{p_i(z)-2}{2}} (D^2_{ij}u)^2\,dx + C''\int_{\Omega}|D_iu|^{p_i(z)+\rho} \,dx +C''
\end{split}
\]
with arbitrary constants $\lambda,\rho\in (0,1)$. Transforming the second term on the left-hand side of \eqref{eq:inter-smooth} as $\mathcal{I}_i$ in  \eqref{eq:3rd-est-pre}, and then plugging into \eqref{eq:inter-smooth} the last two inequalities with sufficiently small $\lambda$ and $\rho$ , we arrive at the differential inequality

\begin{equation}
\label{eq:improved-1}
\begin{split}
\dfrac{1}{2}\dfrac{d}{dt}\|\nabla u\|_{2,\Omega}^2 & +\sum_{i,j=1}^N\int_{\Omega}(\epsilon^2+(D_iu)^2)^{\frac{p_i(z)-2}{2}}(D^2_{ij}u)^2\,dx
\\
& \leq C+ C'\|\nabla u\|_{2,\Omega}^2 + C''\sum_{i=1}^N\int_{\Omega}|D_iu|^{p_i+\rho} \,dx + \int_{\Omega}|\nabla f|^2\,dx.
\end{split}
\end{equation}
Inequality \eqref{eq:4-est} follows after integration in $t$.
\end{proof}
The proof of higher integrability of the gradient mimics the proof given in the case of a rectangular domain. The geometry of the domain $\Omega$ is important in the proof of Proposition \ref{pro:choice-1} where the anisotropic interpolation inequality \eqref{eq:global-conv} is employed. To apply this inequality, we take the smallest rectangular domain $K_{\vec c}$ that contains $\Omega$, and consider the zero continuation of $u$ from $\Omega$ to $K_{\vec c}$ with $p_i(z)=2$ in $K_{\vec c}\setminus \Omega$.

\begin{lemma}
\label{le:high-integr-par-smooth}
Assume that $\partial \Omega\in C^k$, $k\geq 1+N\left(\frac{1}{2}-\frac{1}{p^+}\right)$,  and the exponents $p_i(z)$ satisfy the conditions of Theorem \ref{th:smooth-domain}. Denote by $L$ the maximal of the Lipschitz constants of $p_i(z)$ in $Q_T$. 
If

\[
\mu=\sup_{Q_T}\frac{p^\vee(z)}{p^\wedge(z)}<1+\frac{1}{N},\quad \text{or}\quad \text{ $p_i(z)\geq 2$ in $\overline{Q}_T$ and $\mu<1+\dfrac{2}{N}$},
\]
then for every smooth function $u$, every number

\[
r\in \left(0, \frac{2N(1-\mu)+4}{N+2}\right),
\]
and every $\delta\in (0,1)$

\[
\int_{Q_T}(\epsilon^2+|D_iu|^2)^{\frac{p_i(z)+r-2}{2}}|D_iu|^2\,dz\leq \delta \sum_{j=1}^{N} \int_{Q_T}(\epsilon^2+|D_iu|^2)^{\frac{p_i(z)-2}{2}}(D^2_{ij}u)^2\,dz+C
\]
with a constant $C$ depending on $\delta$, $L$, and $\operatorname{ess}\sup_{(0,T)}\|u\|_{2,\Omega}$.
\end{lemma}

Since the rest of the proof of Theorems \ref{th:existence-reg}, \ref{th:high-diff-reg}, \ref{th:existence}, \ref{th:high-order-reg} is independent of the geometry of $\Omega$, the assertions of Theorem \ref{th:smooth-domain} with $\partial \Omega \in C^k$ follows by a literal repetition of the proofs of the corresponding assertions in the case of a rectangular domain.

\bibliographystyle{siam} 
\bibliography{main}
\end{document}